\documentclass[12pt]{gtart}

\usepackage{amsmath, amssymb, latexsym, amscd, graphicx, epstopdf}
\usepackage{euscript}
\usepackage{enumitem}
\usepackage{subfigure}
\usepackage{epsfig}
\usepackage{psfrag}
\usepackage[makeroom]{cancel}
\usepackage{caption}
\usepackage[inner=30mm, outer=30mm, textheight=225mm]{geometry}

\usepackage[colorlinks,citecolor=blue,linkcolor=blue, backref=page]{hyperref}
\usepackage[nameinlink]{cleveref}

\usepackage{multicol}

\usepackage[colorinlistoftodos]{todonotes}
\makeatletter
\newtheorem*{rep@theorem}{\rep@title}
\newcommand{\newreptheorem}[2]{%
\newenvironment{rep#1}[1]{%
 \def\rep@title{#2 \ref{##1}}%
 \begin{rep@theorem}}%
 {\end{rep@theorem}}}
\makeatother

\newtheorem{theorem}{Theorem}
\newreptheorem{theorem}{Theorem}
\newtheorem{corollary}[theorem]{Corollary}
\newtheorem{lemma}[theorem]{Lemma}
\newtheorem{algorithm}[theorem]{Algorithm}

\theoremstyle{definition}
\newtheorem{defn}[theorem]{Definition}

\newtheorem{claim}[theorem]{Claim}

\newtheoremstyle{TheoremNum}
        {14pt plus6.3pt minus6.3pt}{7.4pt plus3pt minus3pt}              
        {\sl}                      
        {}                              
        {\bf}                     
        {}                             
       {0.75em}                             
        {\thmname{#1}\thmnote{ \bfseries #3}}
    \theoremstyle{TheoremNum}

\newcommand{\R}{\mathbb{R}}
\newcommand{\regina}{{\tt Regina}\rm}

\newcommand{\tri}{\mathcal{T}}

\newcommand{\Z}{\mathbb{Z}}

\DeclareMathOperator{\wt}{wt}

\def\comp{{\backslash\kern-.32em\backslash}}

\begin{document}
\title[Computing closed essential surfaces]{Computing closed essential surfaces in 3--manifolds}
\author{Benjamin A. Burton and Stephan Tillmann}

\begin{abstract}
We present a practical algorithm to test whether a 3--manifold given by a triangulation or an ideal triangulation contains a closed essential surface.
This property has important theoretical and algorithmic consequences.
As a testament to its practicality, we run the algorithm over
a comprehensive body of closed 3--manifolds and knot exteriors, yielding results that were not previously known.
    
The algorithm derives from the original Jaco-Oertel framework,
involves both enumeration and optimisation procedures, and combines
several techniques from normal surface theory. Our methods are relevant for other
difficult computational problems in 3-manifold theory, such as the recognition problem for knots, links and 3--manifolds.
\end{abstract}

%
%

\maketitle

\subsection*{Disclaimer}

The \emph{Journal of Applied and Computational Topology} published a version of this paper on 27 June 2025 that was not approved by the authors and can be accessed at this link:
\begin{center}
{\tt https://doi.org/10.1007/s41468-025-00208-w}\rm
\end{center}
The production process of Springer Nature introduced several errors that were not corrected. These mainly pertain to references (including authors and titles) and how they are cited. Readers are advised to refer to the most current arXiv version of this paper. 

\section{Introduction}

In the study of 3--manifolds, essential surfaces have been of central
importance since Haken's seminal work in the 1960s. 
An essential surface may be regarded as `topologically minimal', and there has since been extensive research  into 3--manifolds, called \emph{Haken 3--manifolds}, that contain an essential surface. The existence of such a surface has
profound consequences for both the topology and geometry of a 3--manifold
\cite{finkelstein99-tubed, haken68-overview, jaco79-jsj-book,
johannson79-book, johannson79-mapping, menasco84-alternating,
oertel84-star, thurston86-deformation, waldhausen68-large}.

Given any closed 3--manifold,
specified by a handle decomposition or triangulation, it is a theorem of Jaco and Oertel
\cite{jaco84-haken} from 1984 that one may algorithmically test for the
existence of a closed essential surface. However, their algorithm
has significant intricacies and is of doubly-exponential complexity in terms of the input size.
One of the key messages of this paper is that this need not be a deterrent: with the right heuristics and careful algorithm engineering, theoretically intractable problems such as this can be implemented and automated over large sets of data.

The issue of iterated-exponential complexity, coming from
cutting and re-triangulating, arises with ubiquity when considering objects
called \emph{normal hierarchies}. These hierarchies are key when solving
more difficult problems (such as the homeomorphism problem) via Haken's approach.
Our strategy in this paper is both fast in
practice and always correct and conclusive, indicating that, despite their iterated-exponential time complexities,
practical implementations of these more difficult
algorithms might indeed be possible.

In the remainder of this introduction, we give a summary of the theoretical results (\Cref{subsec:summary-theory}), an overview of the algorithms (\Cref{subsec:Overview of the algorithms}), and a summary of the computational results (\Cref{subsec:summary-comp}).


\subsection{Summary of the theoretical results}
\label{subsec:summary-theory}

In the short conference paper \cite{burton12-large} written with Alexander Coward, we outlined a practical implementation of the Jaco-Oertel framework for the case of
finding a closed essential surface in a knot exterior in $S^3,$ specified by an ideal triangulation. There, we introduced the bare minimum of the new theory and algorithms required for this application, not including full proofs of many of the results stated. In this paper, we generalise our results and algorithms to triangulations of closed, irreducible, orientable 3--manifolds and to ideal triangulations of the interiors of compact, irreducible and $\partial$--irreducible, orientable 3--manifolds with non-empty boundary. We thank Alexander Coward for his contributions in the early stages of this project.

We base our work on the framework of the Jaco-Oertel algorithm for
testing for closed incompressible surfaces.  This uses
\emph{normal surfaces}, which allow us to translate topological
questions about surfaces into the setting of integer and linear programming.
The framework consists of two stages: the first constructs a finite list of
candidate essential surfaces, and the second tests each surface in the
list to see if it is essential. 

For the {first stage} (enumerating candidate essential surfaces), we combine several techniques. 

First, we wish to create a triangulation for each manifold that contains as few tetrahedra as possible. This is achieved for closed 3--manifolds through the use of \emph{singular triangulations} instead of simplicial ones. For compact manifolds with non-empty boundary we use \emph{ideal triangulations}, which are decompositions of the interiors of these spaces into tetrahedra with their vertices removed. Ideal triangulations introduce some theoretical difficulties, but they are much smaller with roughly half as
many tetrahedra.

Second, we use a variant of normal surface theory based on \emph{quadrilateral coordinates}.
Instead of the standard coordinates for normal surfaces involving four triangles and three quadrilaterals per tetrahedron, we work in the more efficient coordinate system only using the quadrilateral coordinates. 
These coordinates were known to Thurston and Jaco
in the 1980s, and first appeared in print in work of
Tollefson~\cite{tollefson98-quadspace}.
In an ideal triangulation $\tri$ with non-spherical vertex links, this coordinate system encodes both closed normal surfaces and spun-normal surfaces, which are properly embedded and non-compact. The coordinate $x(F)$ of a closed normal surface $F$ in this \emph{spun-normal surface cone $Q(\tri)$} lies in its intersection $Q_0(\tri)$ with a subspace corresponding to the kernel of a \emph{boundary map} (see \Cref{subsec:Boundary map}). 

The following result, proven in \Cref{subsec:Least weight essential surfaces,subsec:proof-of-J-O}, is based on the seminal work of Jaco and Oertel~\cite{jaco84-haken} and provides our finite, constructible set of candidate surfaces. An analogous result was proven by Tollefson for simplicial triangulations of closed 3--manifolds. The concept of a \textbf{mixed triangulation} unifies (possibly singular) triangulations of closed 3--manifolds and ideal triangulations of the interior or a compact manifold with non-empty boundary: here, we allow both \textbf{material} vertices\footnote{The term \emph{material} was coined by Henry Segerman.} (whose links are spheres) and \textbf{ideal} vertices (whose links are orientable surfaces of positive genus).

\begin{theorem}\label{thm-mixed:some extremal is closed essential}
Suppose $M$ is the interior of an irreducible and $\partial$--irreducible, compact, orientable 3--manifold with (possibly empty) boundary, and let $\tri$ be a mixed triangulation of $M.$ If $M$ contains a closed essential surface $S,$ then there is a normal closed essential surface $F$ with the property that $x(F)$ lies on an extremal ray of $Q_0(\tri).$ 

Moreover, if $\chi(S)<0$, then there is such $F$ with $\chi(F)<0$. If, in addition, the link of each ideal vertex has zero Euler characteristic, then if $\chi(S)=0$, then there is such $F$ with $\chi(F)=0.$
\end{theorem}

The hypothesis of the theorem implies that the link of a vertex in $\tri$ is a sphere if and only if the vertex is material. For convenient referencing, we restate the above result in the cases where either all vertices are material or all vertices are ideal. These are the results required for the algorithms of this paper.

\begin{corollary}\label{thm:Tollefson-singular}
Let $M$ be a triangulated, closed, orientable, irreducible 3-manifold. If there exists an essential surface $S$ in $M$, then there is a normal essential surface $F$ with the property that $x(F)$ lies on an extremal ray of $Q_0(\tri)=Q(\tri).$ 

Moreover, if $\chi(S)<0$ (resp.\thinspace $\chi(S)=0$), then there is such $F$ with $\chi(F)<0$ (resp.\thinspace $\chi(F)=0$).
\end{corollary}

\begin{corollary}\label{thm:some extremal is closed essential}
Suppose $M$ is the interior of an irreducible and $\partial$--irreducible, compact, orientable 3--manifold with non-empty boundary, and let $\tri$ be an ideal triangulation of $M.$ If $M$ contains a closed, essential surface $S,$ then there is a normal, closed essential surface $F$ with the property that $x(F)$ lies on an extremal ray of $Q_0(\tri).$ 

Moreover, if $\chi(S)<0$, then there is such $F$ with $\chi(F)<0$. If, in addition, the link of each ideal vertex has zero Euler characteristic, then if $\chi(S)=0$, then there is such $F$ with $\chi(F)=0.$
\end{corollary}

The example in \Cref{subsec:example} of a triangulation $\tri_M$ of a trivial circle bundle over a once-punctured surface of genus two shows that \Cref{thm:some extremal is closed essential} is not correct if $Q_0(\tri)$ is replaced by $Q(\tri).$ It also shows that our approach in \Cref{thm:some extremal is closed essential} is optimal in the following sense. 
Algorithms involving normal surfaces need to reduce the search space to a finite constructible set of solutions in a cone. Typically, one chooses the set of extremal or the set of fundamental solutions. All fundamental solutions of $Q(\tri_M)$ are extremal, and they are either spun-normal surfaces or so-called \emph{thin-edge links}, that is, closed surfaces that after one compression give a boundary parallel torus. In particular, no essential torus is amongst the fundamental solutions in $Q(\tri_M).$ Hence it is necessary to consider the space $Q_0(\tri_M).$ We thank Mark Bell for creating $\tri_M$ for us with {\tt flipper}~\cite{flipper}.


\subsection{Overview of the algorithms}
\label{subsec:Overview of the algorithms}

We describe our algorithms to decide the existence of closed essential surfaces in \Cref{sec:Detecting compressing discs}.
A key difficulty with the Jaco-Oertel framework, which our algorithms also inherit,
is that both stages have running times that are worst-case exponential in their respective input sizes. Moreover, the
output of the first stage (enumerating candidate essential surfaces) is exponential in its input, and this then becomes
the input to the second stage (testing whether a candidate surface is essential). This means that
combining the two stages in any obvious way leads to a doubly-exponential
time complexity solution.

Despite this significant hurdle, we introduce several innovations that cut
down the running time enormously for both stages. Our optimisation for the
first stage involves a combination of established techniques that, though well
understood individually, require new ideas and theory in order to
work harmoniously together.
For the second stage we combine branch-and-bound techniques from
integer programming with the Jaco-Rubinstein procedure for crushing
surfaces within triangulations,
extending recent work of the first author and Ozlen \cite{burton12-unknot}.
 
 The innovations for the {second stage} (testing whether a candidate surface is essential) are of particular significance,
 since there has never before been a systematic algorithm for
 testing whether a candidate surface is essential that is both practical
 and always conclusive. 
Here, the Jaco-Oertel approach cuts along each candidate surface and
inspects the boundary of the resulting 3--manifold to see if it admits a
\emph{compression disc} (such a disc certifies that a surface is non-essential).
The key difficulty is that one requires a new triangulation for the cut-open
3--manifold: since the candidate surface may be very complicated,
any natural scheme for cutting and re-triangulating yields a new triangulation
with exponentially many tetrahedra in the worst case, taking us far
beyond the realm in which normal surface theory has traditionally been
feasible in practice.
Since these new triangulations are the input for stage two,
which is itself exponential time, we now see where the double
exponential arises, and why the Jaco-Oertel framework has long been
considered far from practical. 

We resolve this significant problem using a blend of techniques.
First, we use strong simplification heuristics to reduce the number of
tetrahedra.  Next, we replace the traditional (and very expensive)
enu\-me\-ra\-tion-based search for compression discs with an
\emph{optimisation} process that maximises Euler characteristic.
This uses the branch-and-bound techniques of \cite{burton12-unknot},
and allows us to quickly focus on a single
candidate compression disc.  We employ the crushing techniques
of Jaco and Rubinstein \cite{jaco03-0-efficiency} to quickly test
whether this is indeed a compression disc, and (crucially) to reduce the
size of the triangulation if it is not.

\subsection{Summary of the computational results}
\label{subsec:summary-comp}

In this paper we present a \emph{practical} algorithm that,
though still doubly-exponential in theory,
is able to systematically test a significant class of 3--manifolds
for the \emph{existence} of a closed essential surface, and is both
efficient in practice and always conclusive.
To illustrate its power, we run this algorithm over a comprehensive
body of input data, yielding computer proofs of new mathematical
results. 

\begin{figure}[h]
\begin{center}
  \includegraphics[width=6.8cm]{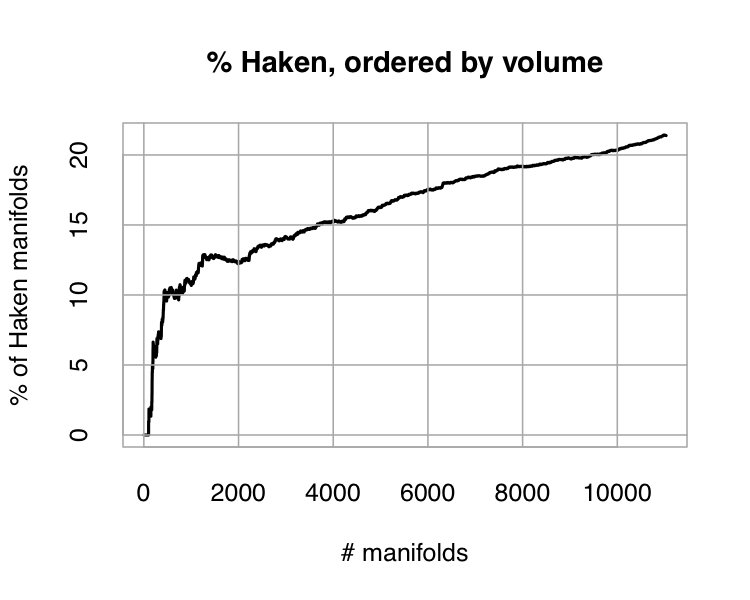}
  \includegraphics[width=6.8cm]{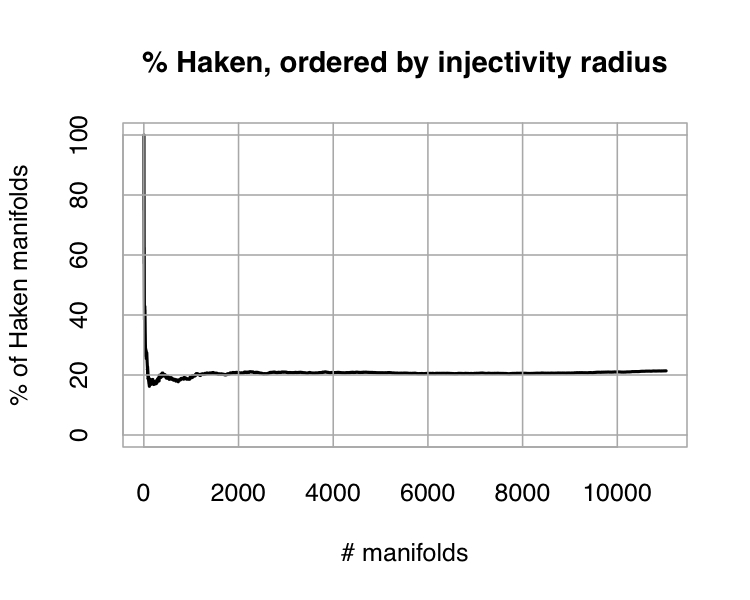}
\end{center}
\caption{Haken manifolds in the Hodgson-Weeks census}
\label{fig:hw-percent}
\end{figure}

We first consider the Hodgson-Weeks census, which contains 11,031 closed, orientable 3--manifolds. This is an approximation to the set of all hyperbolic 3--manifolds of volume $\le 6.5$ and length of shortest geodesic $\ge 0.15.$ The number of tetrahedra ranges from $9$ to $32$. The first step in our algorithm is to check whether the existence of a closed essential surfaces already follows from the fact that the first Betti number is positive. Only 132 of the census manifolds ($\sim$1\%) have positive Betti number, and hence are Haken for this reason. Dunfield used an implementation of the Jaco-Oertel algorithm in 1999 to compute that only 15 of the first 246 census manifolds (246/15 $\sim$ 6\%) are Haken. Our computation gave the surprising result that the percentage of Haken manifolds in the Hodgson-Weeks census is about 21\%, see \Cref{fig:hw-percent}. Further analysis of the computation is given in \Cref{subsec:HW census}.

\begin{figure}[h]
    \centering
     \subfigure[Cumulative number of knots by volume]{%
        \label{fig:htw-byvol-large}%
  \includegraphics[width=6cm]{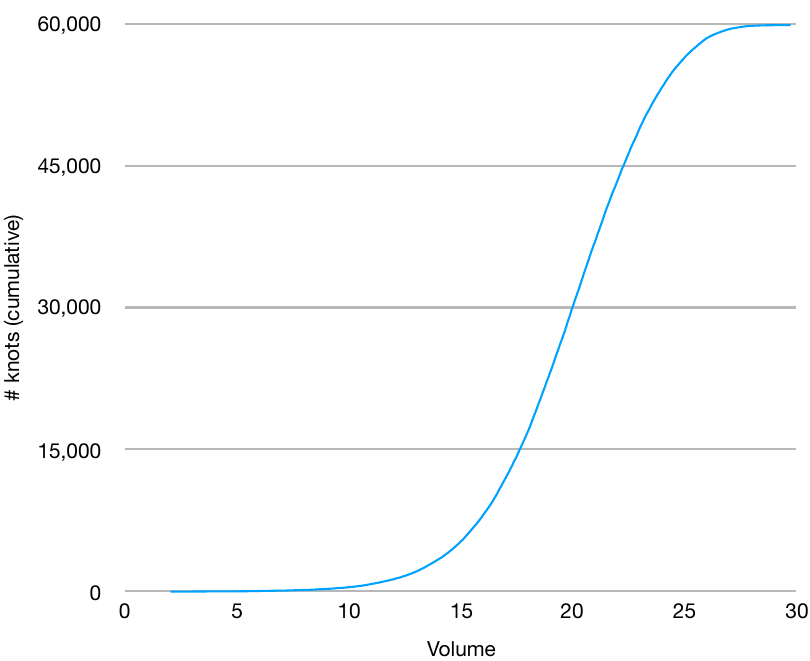}}
  \qquad\qquad
      \subfigure[Cumulative percentage of large knots]{%
        \label{fig:htw-byvol-knotcount}%
  \includegraphics[width=6cm]{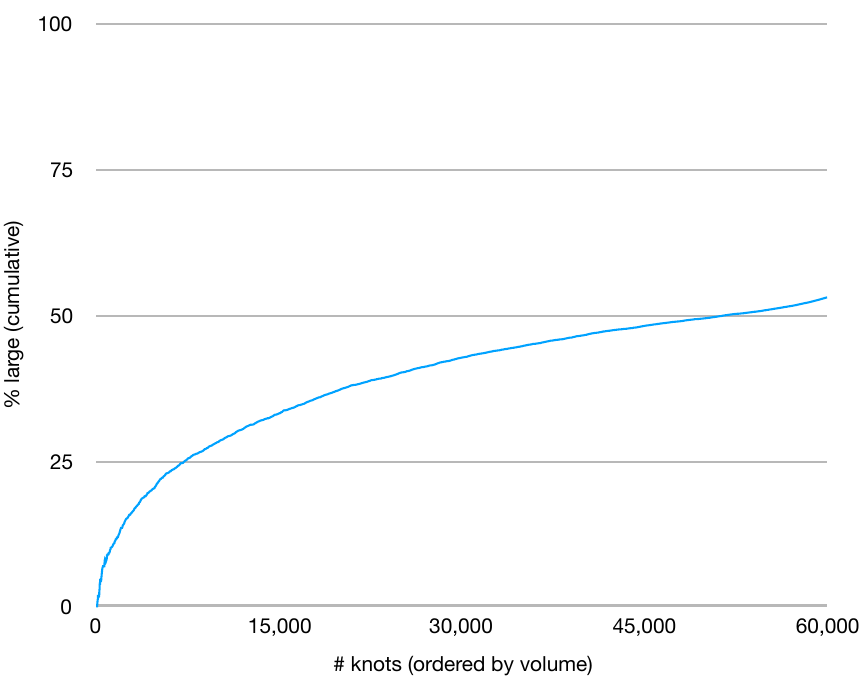}}
  \caption{Percentage of large knots by volume}
\label{fig:htw-byvol}
\end{figure}

The next census we consider is the census of all knots in the 3--sphere with at most 14 crossings due to Hoste-Thistlethwaite-Weeks~\cite{hoste98-first}. The authors thank Morwen Thistlethwaite for sharing these data with them, sorted into torus knots, satellite knots and hyperbolic knots. There are 59,924 hyperbolic knots with at most 14 crossings, and the number of ideal tetrahedra used to triangulate their complements ranges from 2 to 33. In this case, any closed essential surface is separating, and thus the algorithmic detection allows no shortcuts via homology arguments. Our computation showed that 31,805 ($\sim$ 53\%) of the hyperbolic knot complements contain closed essential surfaces (we call these \emph{large knots} for brevity). \Cref{fig:htw-byvol} shows a cumulative plot of the percentage, where the knots in the census are ordered by volume.
Further analysis of the computation is given in \Cref{subsec: knots computation}.

Testament to the improvement in algorithm design is the fact that {\tt Regina}\rm~\cite{regina} can certify the Weber-Seifert dodecahedral space to be non-Haken out of the box in under 75 minutes on a 2016 MacBook Pro with a 3.1 GHz Intel Core i5 processor and 16 GB of 2133 MHz memory. This was a computational challenge due to Thurston when it was first solved in 2012~\cite{burton12-ws}. We caution the reader that running times are subject to random seeds and can vary significantly.


\section{Preliminaries}

In this section, we collect the following well-known definitions and results: singular (possibly ideal) triangulations (\Cref{subsec:Triangulations}); essential surfaces (\Cref{subsec:essential}); a homology criterion for existence of closed essential surfaces (\Cref{subsec:homology criterion}); normal surface theory via quadrilateral coordinates (\Cref{subset:normal surface theory}); reduced form for Haken sums (\Cref{subsec:Reduced form}); the Jaco-Rubinstein method of crushing of triangulations (\Cref{subsec:Crushing}).

\subsection{Triangulations}
\label{subsec:Triangulations}

The notation of \cite{jaco03-0-efficiency} and \cite{tillmann08-finite} will be used in this paper. A \textbf{material triangulation}, $\tri,$ of a compact 3--manifold $M$ consists of a union of pairwise disjoint standard Euclidean 3--simplices, $\widetilde{\Delta} =  \sigma_1\cup \ldots\cup \sigma_t,$ a set of face pairings, $\Phi,$ and a homeomorphism $h\co \widetilde{\Delta} / \Phi \to M.$ We thus make the identification $\widetilde{\Delta} / \Phi = M$ and there is a natural quotient map $p\co \widetilde{\Delta} \to \widetilde{\Delta} / \Phi = M.$ The quotient map is injective on the interior of each 3--simplex. We refer to the image of a 3--simplex in $M$ as a \textbf{tetrahedron} and to \emph{its} faces, edges and vertices with respect to the pre-image. Similarly for images of 2--, 1-- and 0--simplices, which will be referred to as \textbf{faces}, \textbf{edges} and \textbf{vertices} in $M$ respectively. 
If an edge is contained in $\partial M,$ then it is termed a \textbf{boundary edge}; otherwise it is an \textbf{interior edge}. These material triangulations are often called \textbf{singular} or \textbf{pseudo-simplicial} in the literature.

If $M$ is the interior of a compact manifold with non-empty boundary, an \textbf{ideal triangulation}, $\tri,$ of $M$ consists of a union of pairwise disjoint 3--simplices, $\widetilde{\Delta},$ a set of face pairings, $\Phi,$ a natural quotient map $p\co \widetilde{\Delta} \to \widetilde{\Delta} / \Phi = P,$ and a homeomorphism $h \co P \setminus P^{(0)} \to M,$ between $M$ and the complement of the 0--skeleton in $P.$ The quotient space $P$ is usually called a \textbf{pseudo-manifold} and referred to as the \textbf{end-compactification} of $M.$

A natural class of triangulations that includes both of the above types of triangulations is the following. Let $M$ be the interior of a compact manifold with (possibly empty) boundary. A \textbf{mixed triangulation}, $\tri,$ of $M$ consists of a union of pairwise disjoint 3--simplices, $\widetilde{\Delta},$ a set of face pairings, $\Phi,$ a natural quotient map $p\co \widetilde{\Delta} \to \widetilde{\Delta} / \Phi = P,$ a subset $V_i \subseteq P^{(0)}$ of the 0--skeleton in $P$ and a homeomorphism $h \co P \setminus V_i \to M,$ between $M$ and the complement of $V_i$ in $P.$ The vertices in $V_i$ are the \textbf{ideal vertices} of $\tri$ and the vertices in $V_m =  P^{(0)} \setminus V_i$ are the \textbf{material vertices} of $\tri.$

If $V_i = \emptyset$, then $\partial M = \emptyset$ and $\tri$ is a material triangulation. If $V_m = \emptyset,$ then $\partial M \neq \emptyset$ and $\tri$ is an ideal triangulation. Note that $V_i$ is determined by its preimage $\Delta^{(0)}_i = p^{-1}(V_i) \subset \Delta^{(0)}.$

For brevity, we will refer to a 3--manifold $M$ imbued with a mixed triangulation $\tri=(\widetilde{\Delta}, \Delta^{(0)}_i, \Phi, h)$ as a
\textbf{triangulated 3--manifold}. If $M$ is orientable, then we choose an orientation on $M$ and assume that all tetrahedra in $M$ are oriented coherently and the tetrahedra in $\widetilde{\Delta}$ are given the induced orientation.

\subsection{Surfaces in 3-manifolds}
\label{subsec:essential}

Let $M$ be a 3--manifold.
A 2--sphere in $M$ is \textbf{compressible} if it bounds a 3--ball embedded in $M.$ Otherwise the 2--sphere is called \textbf{incompressible}. A 3--manifold in which every 2--sphere is compressible is called \textbf{irreducible}. 

Let $S\subset M$ be a connected surface that is not homeomorphic to the 2--sphere. We say that $S$ is \textbf{compressible} in $M$ if there is a disc $D\subset M$ such that $D \cap S = \partial D$ and $\partial D$ is homotopically non-trivial in $S$ (i.e. does not bound a disc on $S$). Such a disc is called a \textbf{compression disc} for $S$. If $S$ has a compression disc, then $\pi_1(S)\to \pi_1(M)$ is not injective. It follows from classical work of Papakyriakopoulos (namely, Dehn's Lemma and the Loop Theorem) that the converse is also true. See \cite{shalen02-representations} for more details. Detecting compression discs is the topic of \Cref{sec:Detecting compressing discs}.

If the connected surface $S\subset M$ is not homeomorphic to the 2--sphere and has no compression disc, then it is called \textbf{incompressible}. We say that a surface in $M$ is incompressible if each of its connected components is incompressible. Otherwise it is compressible.

The 3--manifold $M$ is \textbf{$\partial$--irreducible} if no boundary component is a 2--sphere and $\partial M$ is incompressible.

The following definition of an essential surface, along with an extensive discussion of its context and ramifications, can be found in Shalen \cite[\S1.5]{shalen02-representations}.

\begin{defn}[(Essential surface)]
A properly embedded surface $S$ in the compact, irreducible, orientable 3--manifold $M$ is \textbf{essential} if it has the following properties:
\begin{enumerate}[itemsep=1mm]
\item $S$ is bicollared, i.e.\thinspace there is a homeomorphism $h$ from $S \times [-1,1]$ onto a neighbourhood of $S$ in $M$ with the property that $h(x,0) = x$ for each $x\in S$ and $h(S \times [-1,1]) \cap \partial M = h(\partial S \times [-1,1]);$
\item the inclusion homomorphism $\pi_1(S_i)\to \pi_1(M)$ is injective for each component $S_i$ of $S$;
\item no component of $S$ is a 2--sphere;
\item no component of $S$ is boundary parallel; and
\item $S$ is non-empty.
\end{enumerate}
\end{defn}

\subsection{Closed essential surfaces from homology}
\label{subsec:homology criterion}

In this paper, we are interested in \emph{closed} essential surfaces in a compact, irreducible, orientable 3--manifold $M$ (given by a triangulation). The universal coefficients theorem and Poincar\'e duality give a natural isomorphism $H_1(M; \Z) \cong H_2(M, \partial M; \Z),$ and it is well-known that every non-trivial class in the latter group is represented by an essential surface in $M$. In particular, if $M$ is closed and $b_1(M)>0,$ then $M$ contains a (necessarily closed) essential surface. Using the intersection pairing and a standard \emph{half-lives-half-dies} argument, one obtains the analogous criterion that if $\partial M\neq \emptyset$ and $2b_1(M)>b_1(\partial M),$ then $M$ contains a closed essential surface. 

Since the calculation of homology with integer coefficients reduces to computing Smith Normal Form of the (generally sparse) boundary matrices, our general approach is to first check if the existence of closed essential surfaces follows from homology. From a computation perspective, this is sensible because the Smith normal form of an integer matrix is computable in polynomial time \cite{Kannan-polynomial-1979}.

\subsection{Normal surface theory}
\label{subset:normal surface theory}

In the case where homology does not certify the existence of a closed essential surface (or is from the outset known not to do so, as is the case for a homology 3--sphere or the complement of a knot or link in a homology 3--sphere), we use Haken's approach to  
connect topology to linear programming via \emph{normal surface theory} in order to search for essential surfaces. 
A \textbf{normal surface} in a (possibly ideal) triangulation
$\tri$ is a properly embedded surface which intersects each
tetrahedron of $\tri$ in a pairwise disjoint collection of
\textbf{triangles} and \textbf{quadrilaterals}, as shown in \Cref{normaldiscs}. These triangles and quadrilaterals are called
\textbf{normal discs}. In an ideal triangulation of a non-compact
3--manifold, a normal surface may contain infinitely many triangles;
such a surface is called \textbf{spun-normal} \cite{tillmann08-finite}.
A normal surface may be disconnected or empty.

\begin{figure}[h!]  
\centering
\includegraphics[scale=1]{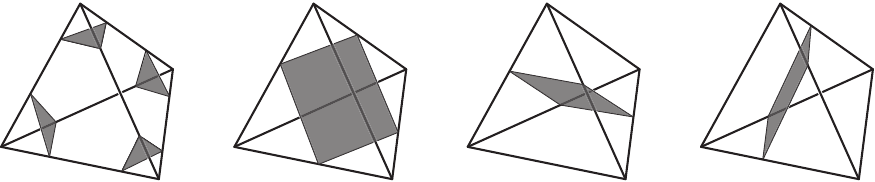}
\caption{The seven types of normal disc in a tetrahedron.} 
\label{normaldiscs}
\end{figure}

We now describe an algebraic approach to normal surfaces.
The key observation is that each normal surface contains
finitely many quadrilateral discs, and is uniquely determined
(up to normal isotopy) by these quadrilateral discs. Here a \textbf{normal
isotopy} of $M$ is an isotopy that keeps all simplices of all dimensions
fixed. Let $\square$ denote the set of all normal isotopy classes
of normal quadrilateral discs in $\tri$,
so that $|\square| = 3t$ where $t$ is the number of tetrahedra in $\tri$.
These normal isotopy classes are called quadrilateral \textbf{types}.

We identify $\R^\square$ with $\R^{3t}.$
Given a normal surface $S,$ let $x(S) \in \R^\square = \R^{3t}$
denote the integer vector for which each coordinate $x(S)(q)$
counts the number of quadrilateral discs in $S$ of type $q \in \square$.
This \textbf{normal $Q$--coordinate} $x(S)$ satisfies the
following two algebraic conditions.

First, $x(S)$ is admissible.
A vector $x \in \R^\square$ is \textbf{admissible} if
$x \ge 0$, and for each tetrahedron $x$ is non-zero
on at most one of its three quadrilateral types. 
This reflects the fact that an embedded surface cannot contain
two different types of quadrilateral in the same tetrahedron.

Second, $x(S)$ satisfies a linear equation for each interior
edge in $M,$ termed a \textbf{$Q$--matching equation}.
Intuitively, these equations arise from the fact that as one
circumnavigates the earth, one crosses the equator from north to south
as often as one crosses it from south to north.
We now give the precise form of these equations.
To simplify the discussion,
we assume that $M$ is oriented and all tetrahedra are given
the induced orientation; see \cite[Section~2.9]{tillmann08-finite} for
details.

\begin{figure}[h]
    \centering
    \subfigure[The abstract neighbourhood $B(e)$]{%
        \label{fig:matchingquadbdry}%
        \includegraphics[scale=1.1]{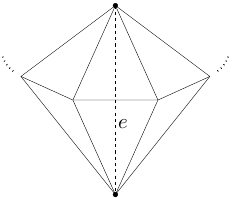}}
    \qquad
    \subfigure[Positive slope $+1$]{%
        \label{fig:matchingquadpos}%
        \includegraphics[scale=1.1]{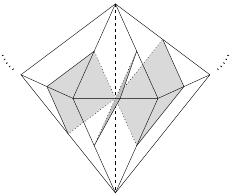}}
    \qquad
    \subfigure[Negative slope $-1$]{%
        \label{fig:matchingquadneg}%
        \includegraphics[scale=1.1]{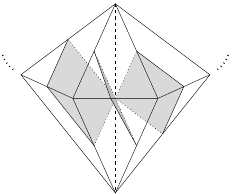}}
    \caption{Slopes of quadrilaterals}
    \label{fig:slopes}
\end{figure}

Consider the collection $\mathcal{C}$ of all (ideal) tetrahedra meeting
at an edge $e$ in $M$ (including $k$ copies of tetrahedron $\sigma$ if $e$ occurs
$k$ times as an edge in $\sigma$).  We form the \textbf{abstract
neighbourhood $B(e)$} of $e$ by pairwise identifying faces of tetrahedra
in $\mathcal{C}$ such that there is a well defined quotient map from
$B(e)$ to the neighbourhood of $e$ in $M$; see
\Cref{fig:matchingquadbdry} for an illustration.  Then  $B(e)$ is a ball
(possibly with finitely many points missing on its boundary).  We think
of the (ideal) endpoints of $e$ as the poles of its boundary sphere, and
the remaining points as positioned on the equator.

Let $\sigma$ be a
tetrahedron in $\mathcal{C}$. The boundary square of a normal
quadrilateral of type $q$ in $\sigma$ meets the equator of $\partial
B(e)$ if and only it has a vertex on $e$. In this case, it has a slope
$\pm1$ of a well--defined sign on $\partial B(e)$ which is independent of the
orientation of $e$. Refer to \Cref{fig:matchingquadpos,fig:matchingquadneg}, which show quadrilaterals with
\textbf{positive} and \textbf{negative slopes} respectively.

Given a quadrilateral type $q$ and an edge $e,$ there is a
\textbf{total weight} $\wt_e(q)$ of $q$ at $e,$ which records the sum of
all slopes of $q$ at $e$ (we sum because $q$ might meet $e$ more than
once, if $e$ appears as multiple edges of the same tetrahedron).
If $q$ has no corner on $e,$ then we set
$\wt_e(q)=0.$ Given edge $e$ in $M,$ the $Q$--matching equation of $e$
is then defined by $0 = \sum_{q\in \square}\; \wt_e(q)\;x(q)$.

\begin{theorem}\label{thm:admissible integer solution gives normal}
For each $x\in \R^\square$ with the properties that $x$ has integral coordinates, $x$ is admissible and
$x$ satisfies the $Q$--matching equations, there is a (possibly
non-compact) normal surface $S$ such that $x = x(S).$ Moreover, $S$ is
unique up to normal isotopy and adding or removing vertex linking surfaces,
i.e., normal surfaces consisting entirely of normal triangles.
\end{theorem}

This is related to Hauptsatz 2 of \cite{haken61-knot}. For a proof of \Cref{thm:admissible integer solution gives normal}, see \cite[Theorem~2.4]{tillmann08-finite}. For restricted classes of manifolds and triangulations, earlier proofs are in \cite[Theorem 1]{tollefson98-quadspace} and \cite[Theorem~2.1]{kang05-spun}.
The set of all $x\in \R^{\square}$ with the property that
(i)~$x\ge 0$ and (ii)~$x$ satisfies the $Q$--matching equations is denoted
$Q(\tri).$ This
naturally is a polyhedral cone, but
the set of all admissible $x\in \R^{\square}$
typically meets $Q(\tri)$ in a non-convex set.

The polyhedral cone $Q(\tri)$ is defined by a finite number of linear inequalities and linear equalities with integer coefficients and has the origin as its cone point. Since $x\ge 0$ for each $x\in Q(\tri),$ the intersection of $Q(\tri)$ with the standard simplex $\{ x\ge 0 \;\mid\; \sum x_i = 1\}$ is a compact polyhedron. The vertices of this polyhedron define the \textbf{extremal rays} of $Q(\tri).$ Each extremal ray contains a point with integer coordinates, and hence the normal $Q$--coordinate of a normal surface. We call a two-sided, connected normal surface $S$ with the property that $x(S)$ lies on an extremal ray of the polyhedral cone $Q(\tri)$ a \textbf{$Q$--vertex surface}.

Tollefson~\cite{tollefson98-quadspace} proved the following theorem building on the work of Jaco and Oertel~\cite{jaco84-haken}.

\begin{theorem}[Tollefson]\label{thm:Tollefson}
Let $M$ be a simplicially triangulated, compact, irreducible, $\partial$--irreducible 3-manifold. If there exists a two-sided, incompressible, $\partial$--incompressible surface in $M$, then there exists one that is a $Q$--vertex surface.
\end{theorem}

We remark that Tollefson (and the authors he cites) work with simplicial triangulations or handlebody decompositions. 
The above theorem leaves open the possibility that the $Q$--vertex surface is boundary parallel. \Cref{thm-mixed:some extremal is closed essential} below provides a more general result for a different class of triangulations of orientable 3--manifolds, which includes the statement of \Cref{thm:Tollefson} in the case $M$ is closed and orientable (see \Cref{thm:Tollefson-singular}).

Every integer vector in $Q(\tri)$ is a linear combination of the normal $Q$--coordinates of the $Q$--vertex surfaces over the non-negative \emph{rational} numbers. This is a technical detail that needs to be addressed in the proof. In some applications of normal surface theory, one requires arguments that only allow non-negative \emph{integer} linear combinations. The polyhedral cone $Q(\tri)$ has a so-called \emph{Hilbert basis} consisting of a finite number of \textbf{fundamental solutions} that have the property that each integer lattice point in $Q(\tri)$ is a linear combination of the fundamental solutions over the non-negative integers. Each fundamental solution defines a \textbf{$Q$--fundamental surface}. We note that a $Q$--vertex surface is either a $Q$--fundamental surface or the boundary of a regular neighbourhood of a 1--sided $Q$--fundamental surface.


\subsection{Haken sum and reduced form}
\label{subsec:Reduced form}

Let $M$ be a (ideally or materially) triangulated 3--manifold.
The \textbf{weight} of the normal surface $F$ is the cardinality of its intersection with the 1--skeleton,  $\wt(F) = |F \cap M^{(1)}|.$ If $F$ is closed, then its weight is finite.

\begin{figure}[h]
\begin{center}
  \subfigure[]{
      \includegraphics[width=9cm]{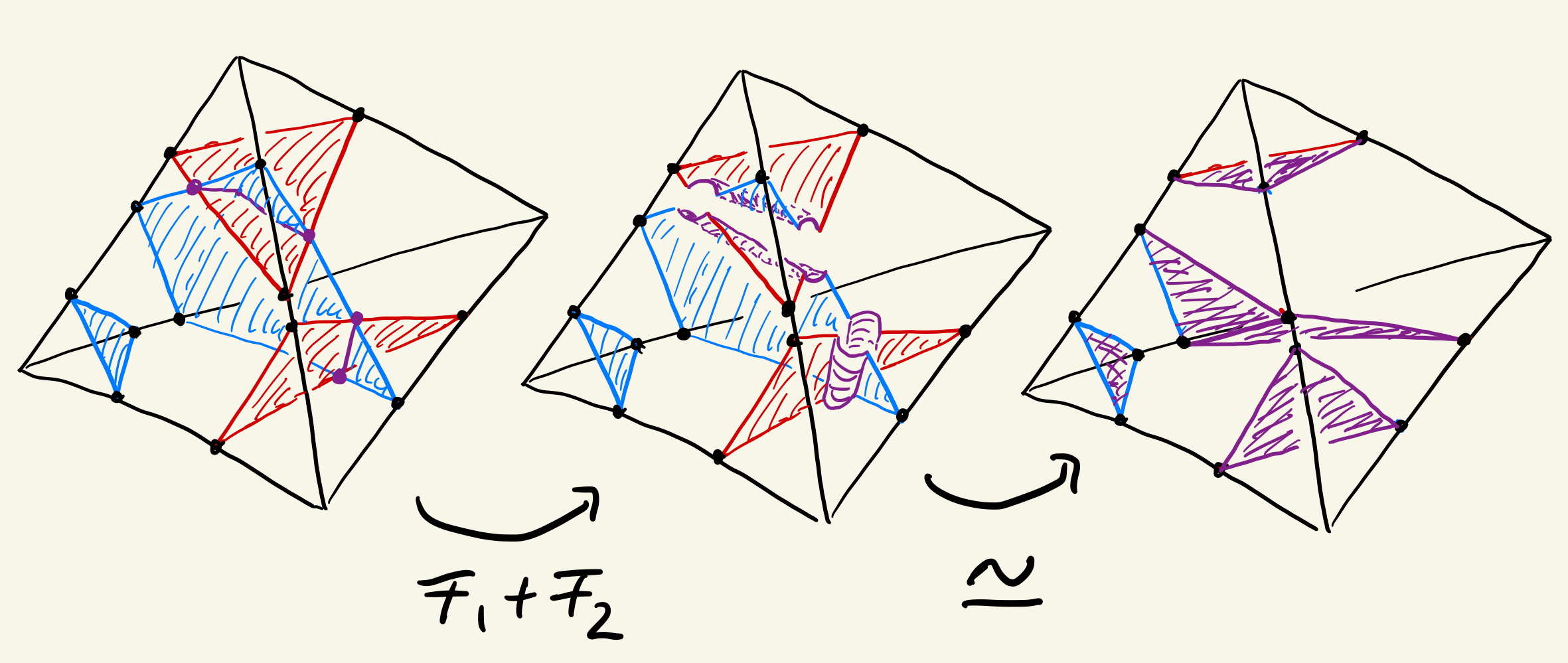}
    } 
\subfigure[]{
      \includegraphics[width=9cm]{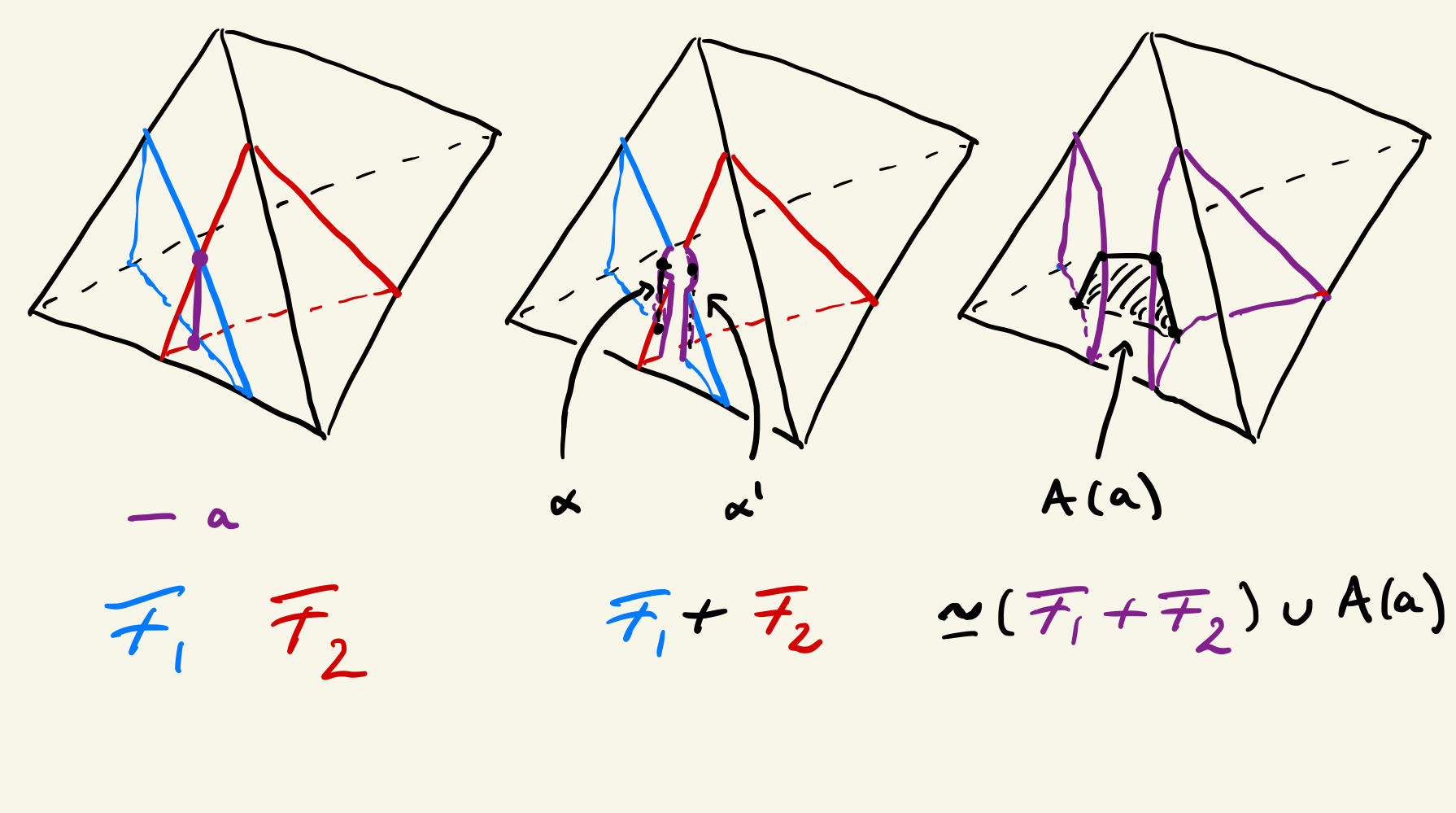}
    } 
\end{center}
    \caption{Regular exchange of normal discs}
     \label{fig:regular exchange of discs}
\end{figure}

Two normal surfaces are \textbf{compatible} if they do not meet a tetrahedron in quadrilateral discs of different types. The normal surfaces $F_1$ and $F_2$ are in \textbf{general position} if:
\begin{enumerate}
\item For each 1--simplex $\Delta^1 \in \tri^{(1)}$, $F_1 \cap F_2 \cap \Delta^{1} = \emptyset.$
\item For each 2--simplex $\Delta^2 \in \tri^{(2)}$ and all normal arcs $f_i \subset F_i \cap \Delta^2,$ we have $f_1 \cap f_2$ is empty or a single point.
\item For each 3--simplex $\Delta^3 \in \tri^{(3)}$ and all normal discs $D_i \subset F_i \cap \Delta^3,$ we have $D_1 \cap D_2$ is empty or a single arc unless $D_1$ and $D_2$ are quadrilateral discs of the same type.
\end{enumerate}
It is not difficult to show that given two normal surfaces, there is a normal isotopy that puts them into general position.

Suppose $F_1$ and $F_2$ are closed normal surfaces that are compatible and in general position.
In this case, one can define a new normal surface, denoted $F_1+F_2$ and called the \textbf{Haken sum of $F_1$ and $F_2$}, as follows.

Denote $N(F_1 \cap F_2)$ a small, open, tubular neighbourhood of $F_1 \cap F_2$ with the property that the neighbourhood $N(a)\subset N(F_1 \cap F_2)$ of each connected component $a$ of $F_1 \cap F_2$ meets each $F_i$ in an annulus or M\"obius band. Also note that $a$ is orientation preserving in $M$ and therefore $N(a)$ is a solid torus. Now $N(a)$ meets $F_1$ in an annulus if and only if it meets $F_2$ in an annulus if and only if $a$ is 2--sided in $F_1$ if and only if $a$ is 2--sided in $F_2.$ Similarly, $N(a)$ meets $F_1$ in a M\"obius band if and only if it meets $F_2$ in a M\"obius band if and only if $a$ is 1--sided in $F_1$ if and only if $a$ is 1--sided in $F_2.$

The connected components of the closure of $(F_1 \cup F_2) \setminus N(F_1 \cap F_2)$ are called \textbf{patches} of $F_1+F_2.$ 

The patches can be connected by annuli on $\partial N(F_1 \cap F_2)$ to form a new surface $F_1+F_2.$ At each intersection curve, there are two choices for such annuli. We make the unique choice determined by the intersection of the corresponding normal discs that ensures that the resulting surface is normal without any further isotopy (see \cite[\S 1.8]{tillmann08-finite} for a proof and \Cref{fig:regular exchange of discs} for some possible configurations). These annuli are called \textbf{regular exchange annuli}, and their complementary annuli are the \textbf{irregular exchange annuli}. On $\partial N(a)$, there is either one of each type of annulus or two of each type, depending on whether $a$ is 1--sided or 2--sided.

Each regular exchange annulus has a product structure of the form (boundary curve of patch)$\times [0,1].$ We call the curves (boundary curve of patch)$\times \frac{1}{2}$ the \textbf{trace curves} on $F_1+F_2$ and the components of the complement of the set of all trace curves on $F_1+F_2$ the \textbf{regions}. Note that each region is a small regular neighbourhood of a unique patch in $F_1+F_2.$ 
A component $a$ of $F_1 \cap F_2$ contributes two trace curves if it is two-sided on both $F_1$ and $F_2$ and it contributes one trace curve if it is one-sided on both $F_1$ and $F_2.$ There is a \textbf{0--weight band} $A(a)$ (which is either an annulus or a M\"obius band) properly embedded in $N(a)$ with boundary the trace curve(s) and core curve the intersection curve $a$. Note that such a 0--weight band only meets the normal surface $F_1+F_2$ in its boundary.

If $a$ is 2--sided, then patches $P_1$ on $F_1$ and $P_2$ on $F_2$ are \textbf{adjacent at $a$} if their boundaries are joined by an irregular exchange annulus associated with $a.$ Moreover, if $C_i \subset F_1+F_2$ are subsurfaces with $P_i \subset C_i,$ the $C_1$ and $C_2$ are \textbf{adjacent at $a$}.

If $F_1$ and $F_2$ do not contain vertex linking surfaces, then they are uniquely determined, up to normal isotopy, by their normal $Q$--coordinates $x(F_1)$ and $x(F_2).$  Then $x(F_1)+x(F_2)$ is an admissible solution to the $Q$--matching equations, and hence represented by a unique closed normal surface $F$ without vertex linking components. Note that the Haken sum $F_1+F_2$ has the same normal $Q$--coordinate as $F$, that is $x(F) = x(F_1+F_2).$ It follows that $F$ is related to $F_1+F_2$ by normal isotopy and deleting any components of $F_1+F_2$ that are vertex linking surfaces.

We may therefore normally isotope $F$ such that we have the identity $F + \Sigma = F_1 + F_2,$ where $\Sigma$ is a (possibly empty) finite union of vertex linking surfaces. The fundamental surfaces are precisely those that cannot be written as a non-trivial Haken sum in the following sense. Suppose that $F$ is a $Q$--fundamental surface and $F + \Sigma = F_1 + F_2.$ Then $x(F) = x(F_1)+x(F_2)$ and hence $F$ being fundamental implies either $x(F_1) = x(F)$ or $x(F_2) = x(F).$ In the first case, $F_1$ contains a component normally isotopic to $F$ and the remaining components of $F_1$ and all components of $F_2$ are vertex linking surfaces. Similarly in the second case.

Both weight and Euler characteristic are additive under Haken sum. So we have
\begin{align*}
\wt(F_1) + \wt(F_2) &=\wt(F) + \wt(\Sigma)\\
\chi(F_1) + \chi(F_2) &=\chi(F) + \chi(\Sigma)
\end{align*}
The sum 
\[
F+\Sigma=F_1 + F_2
\]
 is said to be in \textbf{reduced form} if there is no Haken sum 
 \[F+\Sigma'=F'_1 + F'_2\]
  where $F'_i$ is isotopic to $F_i$ in $M,$ $F'_1 \cap F'_2$ has fewer components than $F_1 \cap F_2$ and $\Sigma'$ is a union of vertex linking surfaces. 
For instance, if $M$ is irreducible and $\partial$--irreducible and $F+\Sigma=F_1 + F_2$ is in reduced form, then there are no two adjacent patches that are both discs.

\subsection{Crushing triangulations}
\label{subsec:Crushing}

The crushing process of Jaco and Rubinstein
\cite{jaco03-0-efficiency} plays an important role in our
algorithms, and we informally outline this process here.
We refer the reader to
\cite{jaco03-0-efficiency} for the formal details,
or to \cite{burton12-crushing} for a simplified approach.

\begin{figure}[ht]
    \centering
    \subfigure[Pieces after cutting open along $S$]{%
        \label{fig:jrslices}%
        \includegraphics[scale=0.6]{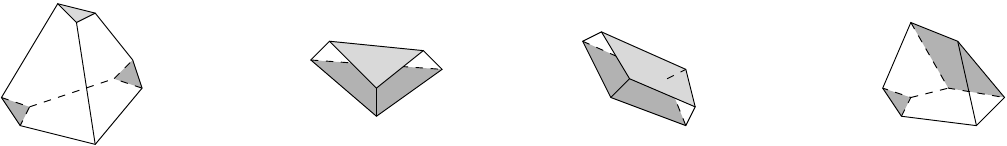}}
    \\
    \subfigure[Pieces after crushing $S$ to a point]{%
        \label{fig:jrpieces}%
        \includegraphics[scale=0.6]{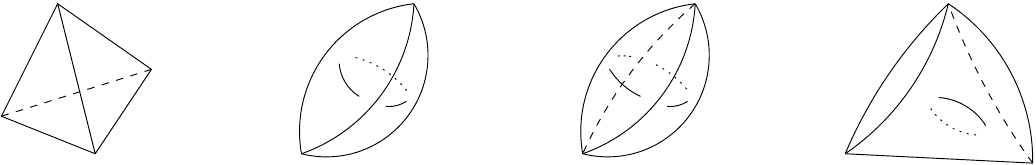}}
    \\
    \subfigure[Flattening footballs and pillows]{%
        \label{fig:jrflatten}%
        \includegraphics[scale=0.9]{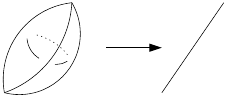}\qquad
        \includegraphics[scale=0.9]{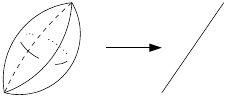}\qquad
        \includegraphics[scale=0.9]{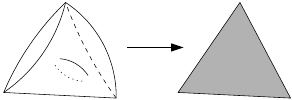}}
    \caption{Steps in the Jaco-Rubinstein crushing process}
    \label{fig-jrcrush}
\end{figure}

Let $S$ be a two-sided normal surface in a triangulation $\tri$ of a compact
orientable 3-manifold $M$ (with or without boundary).
To \textbf{crush} $S$ in $\tri$, we
(i)~cut $\tri$ open along $S$, which splits each tetrahedron
into a number of (typically non-tetrahedral) pieces,
several of which are illustrated in \Cref{fig:jrslices};
(ii)~crush each resulting copy of $S$ on the boundary to a point,
which converts these pieces into tetrahedra, footballs and/or pillows
as shown in \Cref{fig:jrpieces}; and
(iii)~flatten each football or pillow to an edge or triangle
respectively, as shown in \Cref{fig:jrflatten}.

The result is a new collection of tetrahedra with a new set of face
identifications.  We emphasise that we \emph{only} keep track of face
identifications between tetrahedra:
any ``pinched'' edges or vertices fall apart, and any
lower-dimensional pieces (triangles, edges or vertices) that do not
belong to any tetrahedra simply disappear.
The resulting structure might not represent a 3-manifold
triangulation, and even if it does the flattening operations might have
changed the underlying 3-manifold in ways that we did not intend.

Although crushing can cause a myriad of problems in general,
Jaco and Rubinstein show that in some cases the operation behaves
extremely well \cite{jaco03-0-efficiency}.  In particular, if
$S$ is a normal sphere or disc, then after crushing we always obtain a
triangulation of some 3-manifold $M'$ (possibly disconnected, and
possibly empty) that is obtained from the original
$M$ by zero or more of the following operations:
\begin{itemize}
\item cutting manifolds open along spheres and filling the resulting boundary spheres with
3-balls;
\item cutting manifolds open along properly embedded discs;
\item capping boundary spheres of manifolds with 3-balls;
\item deleting entire connected components that are any of the 3-ball, the
3-sphere,
projective space $\R P^3,$ the lens space $L(3,1)$ or the product space $S^2 \times S^1.$
\end{itemize}

An important observation is that the number of tetrahedra that remain
after crushing
is precisely the number of tetrahedra that do not contain quadrilaterals of $S$.

\section{Closed normal surfaces in $Q$--space}
\label{sec:closed}

In this section we review the linear \emph{boundary map} of \cite{tillmann08-finite},
with which we restrict the normal surface solution space to closed surfaces only, and we provide the required extensions of Jaco and Oertel's result in the context of singular triangulations and ideal triangulations. Throughout this section, we assume that $M$ is the interior of a compact, orientable manifold with (possibly empty) boundary.

\subsection{Boundary map}
\label{subsec:Boundary map}

Suppose $\tri$ is a mixed triangulation of $M.$ 
The link of a (material or ideal) vertex $v$ is an orientable surface $B_v$ of genus $g_v\ge 0$, and we may assume that $B_v$ is a normal surface entirely made up of normal triangles. Let $\gamma \in H_1(B_v; \R).$ We now describe an associated linear functional $\nu(\gamma) \co \R^{\square}\to \R,$ which measures the behaviour along $\gamma$ of a normal surface near the ideal vertex $v$. The idea is similar to the intuitive description of the $Q$--matching equations. As one goes along $\gamma$ and looks down into the manifold, normal quadrilaterals will (as Jeff Weeks puts it) \emph{come up from below} or \emph{drop down out of sight}. If the total number coming up minus the total number dropping down is non-zero, then the surface spirals towards the vertex in the cross section $\gamma \times [0, \infty) \subset B_v \times [0, \infty)$ and the sign indicates the direction, see Figure~\ref{fig:spinning}(b) for a sketch when $B_v$ is a torus. If this number is zero, then after a suitable isotopy the surface meets the cross section in a (possibly empty) union of circles, see \Cref{fig:spinning}(c).

\begin{figure}[h]
  \begin{center}
    \subfigure[$0=\nu_x(\gamma) = \sum_{i=1}^{k} (-1)^i x(q_i)$ is the $Q$--matching equation]{
      \includegraphics[scale=1]{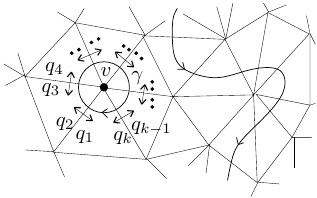}}
    \qquad
    \subfigure[spun $\Longleftrightarrow$ $\nu_x\neq 0$]{
      \includegraphics[scale=0.9]{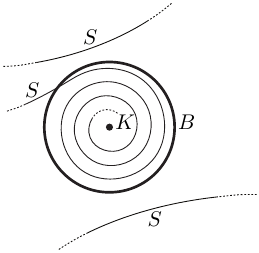}}
    \subfigure[not spun $\Longleftrightarrow$ $\nu_x= 0$]{
      \includegraphics[scale=0.9]{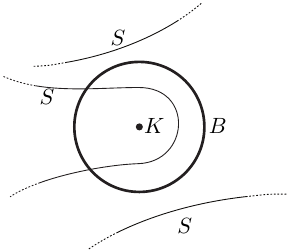}}
  \end{center}
  \caption{Boundary map determines $Q$--matching equations and spinning}
  \label{fig:spinning}
\end{figure}

The surface $B_v$ has an induced triangulation consisting of normal triangles. Represent $\gamma$ by an oriented path on $B_v,$ which is disjoint from the 0--skeleton and meets the 1--skeleton transversely. Each edge of a triangle in $B_v$ is a normal arc. Give the edges of each triangle in $B_v$ transverse orientations pointing into the triangle and labelled by the quadrilateral types sharing the normal arc with the triangle; see \Cref{fig:quad coods}. We then define $\nu(\gamma)$ as follows. Choosing any starting point on $\gamma,$ we read off a formal linear combination of quadrilateral types $q$ by taking $+q$ each time the corresponding edge is crossed with the transverse orientation, and $-q$ each time it is crossed against the transverse orientation (where each edge in $B_v$ is counted twice---using the two adjacent triangles). 

\begin{figure}[h]
    \centering
    \includegraphics[scale=1.1]{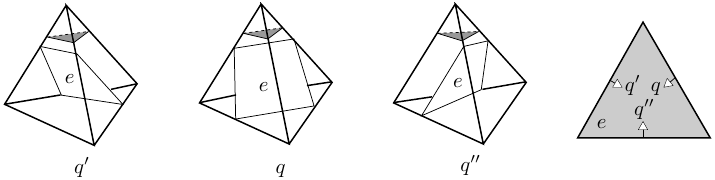}
    \caption{Coming up and dropping down}
    \label{fig:quad coods}
\end{figure}

Evaluating $\nu(\gamma)$ at some $x\in \R^{\square}$ gives a real number $\nu_{x}(\gamma).$ For example, taking a small loop around a vertex in $T$ and setting this equal to zero gives the $Q$--matching equation of the corresponding edge in $M;$ see \Cref{fig:spinning}(a). For each $x\in Q(\tri),$ the resulting map $\nu_{v,x}\co H_1(B_v ; \R) \to \R$ is a well-defined homomorphism, which has the property that the surface in Theorem~\ref{thm:admissible integer solution gives normal} is closed if and only if $\nu_{v,x} = 0$ (see \cite{tillmann08-finite}, Proposition~3.3). Since $\nu_{v,x}\co H_1(B_v; \R) \to \R$ is a homomorphism, it is trivial if and only if we have $\nu_{v,x}(\alpha_i) = 0 = \nu_{v,x}(\beta_i)$ for any basis $\{\alpha_i, \beta_i\}_{1\le i\le g_v}$ of $H_1(B_v; \R).$

We define $\nu_x = \oplus_v\; \nu_{v,x},$ where the sum is taken over all vertices. The surface in \Cref{thm:admissible integer solution gives normal} is closed if and only if $\nu_x =0$ (see \cite{tillmann08-finite}, Proposition 3.3). We then define 
\[
Q_0(\tri) = Q(\tri) \cap \{ x \mid \nu_x=0\}
\]  and call a 2--sided, connected normal, surface $F$ with $x(F)$ on an extremal ray of $Q_0(\tri)$ a \textbf{$Q_0$--vertex surface.}

We remark that if $B_v$ is a sphere, then $\nu_{v,x} = 0$ and hence in the case where each $B_v$ is a sphere, we have $Q_0(\tri) = Q(\tri).$

\begin{lemma}\label{lem:sum_of_verices}
Let $F$ be a closed normal surface. Then there are $Q_0$--vertex surfaces $V_i$ and $n, n_i \in \mathbb{N}\setminus \{0\}$ such that $n x(F) = \sum n_i x(V_i).$ In particular, we can write $nF + \Sigma = \sum n_i V_i,$ where $\Sigma$ is a union of vertex linking surfaces and each intersection curve in the Haken sum is 2--sided.
\end{lemma}

\begin{proof}
We have $$m x(F) = \sum m_i x(V'_i),$$ where $m, m_i \in \mathbb{N}\setminus \{0\}$ and either $V'_i$ or $2V'_i$ is a $Q_0$--vertex surface for each $i.$ The two cases arise from the fact that we require a $Q_0$--vertex surface to be 2--sided and connected: If $V$ corresponds to the first integer lattice point on an admissible extremal ray of $Q_0(\tri)$ and $V$ is 1--sided, then the corresponding $Q_0$--vertex surface is $2V,$ obtained by taking the boundary of a regular neighbourhood of $V.$ If $V'_i$ is a $Q_0$--vertex surface, let $n_i = 2m_i$ and $V_i=V'_i.$
If $2V'_i$ is a $Q_0$--vertex surface, let $n_i = m_i$ and $V_i=2V'_i.$ Finally let $n=2m.$ Then $n x(F) = \sum n_i x(V_i)$ and each $V_i$ is a $Q_0$--vertex surface.
\end{proof}

We have now introduced all notation and definitions required for our generalisation of \cite[Theorem 2.2]{jaco84-haken}:

\begin{reptheorem}{thm-mixed:some extremal is closed essential}
Suppose $M$ is the interior of an irreducible and $\partial$--irreducible, compact, orientable 3--manifold with (possibly empty) boundary, and let $\tri$ be a mixed triangulation of $M.$ If $M$ contains a closed, essential surface $S,$ then there is a normal, closed essential surface $F$ with the property that $x(F)$ lies on an extremal ray of $Q_0(\tri).$ 

Moreover, if $\chi(S)<0$, then there is such $F$ with $\chi(F)<0$. If, in addition, the link of each ideal vertex has zero Euler characteristic, then if $\chi(S)=0$, then there is such $F$ with $\chi(F)=0.$
\end{reptheorem}

The proof is given in the next two sections. It is an adaptation of the proofs of \cite[Lemma 2.1]{jaco84-haken} and \cite[Theorem 2.2]{jaco84-haken}. The adaptation involves both moving from normal surfaces with respect to handle decompositions to normal surfaces with respect to triangulations, and moving from standard coordinates to quadrilateral coordinates. The former creates almost no differences, whilst the latter requires some extra techniques, mainly to address vertex links that are spheres or the possibility of vertex links that are not of minimal weight in their isotopy classes. In particular, we wanted to avoid assuming that ideal edges are essential.

Based on feedback by a referee, we not only describe the modifications  to the proofs in \cite{jaco84-haken} that are required, but rather give a complete proof. The arguments that are repeated almost verbatim from \cite{jaco84-haken} are indented, and we hope they serve as an invitation to read the original. There are two stages: \Cref{subsec:Least weight essential surfaces} contains auxiliary results about Haken sums giving least weight closed essential normal surfaces. The proof of \Cref{thm-mixed:some extremal is closed essential} is then assembled in \Cref{subsec:proof-of-J-O}.


\subsection{Least weight essential normal surfaces}
\label{subsec:Least weight essential surfaces}

Suppose $M$ is the interior of an irreducible and $\partial$--irreducible, compact, orientable 3--manifold with (possibly empty) boundary, and let $\tri$ be a mixed triangulation of $M.$

Suppose $M$ contains the closed, essential, normal surface $F.$ Replace this by a normal surface that has least weight amongst all normal surfaces isotopic (but not necessarily normally isotopic) to $F.$ We let $F$ denote this surface.

We let $nF$ denote the normal surface obtained by taking $n$ parallel copies of $F.$ Clearly, $x(nF) = nx(F),$ and since $F$ has least weight in its isotopy class, so does $nF$ because $F$ is 2--sided. To sum up, $nF$ is a closed, essential, normal surface which has least weight amongst all normal surfaces in its isotopy class.

Suppose $F + \Sigma = F_1 + F_2$ is in reduced form, where each $F_i$ is a closed normal surface no component of which is a vertex linking surface and $\Sigma$ a union of vertex linking surfaces. If some component of $F_1 \cap F_2$ is 1--sided in $F_1$ (and hence $F_2$), replace $F$ by $2F$, $\Sigma$ by $2\Sigma$, and $F_i$ by $2F_i$. The surface $2F$ is also essential and least weight in its isotopy class. Hence, we may assume that each intersection curve in $F + \Sigma = F_1 + F_2$ is 2--sided and that all surfaces in the sum are 2--sided. In particular, each component of $F_1 \cap F_2$ contributes two trace curves.

\begin{figure}[h]
    \centering
    \includegraphics[width=10cm]{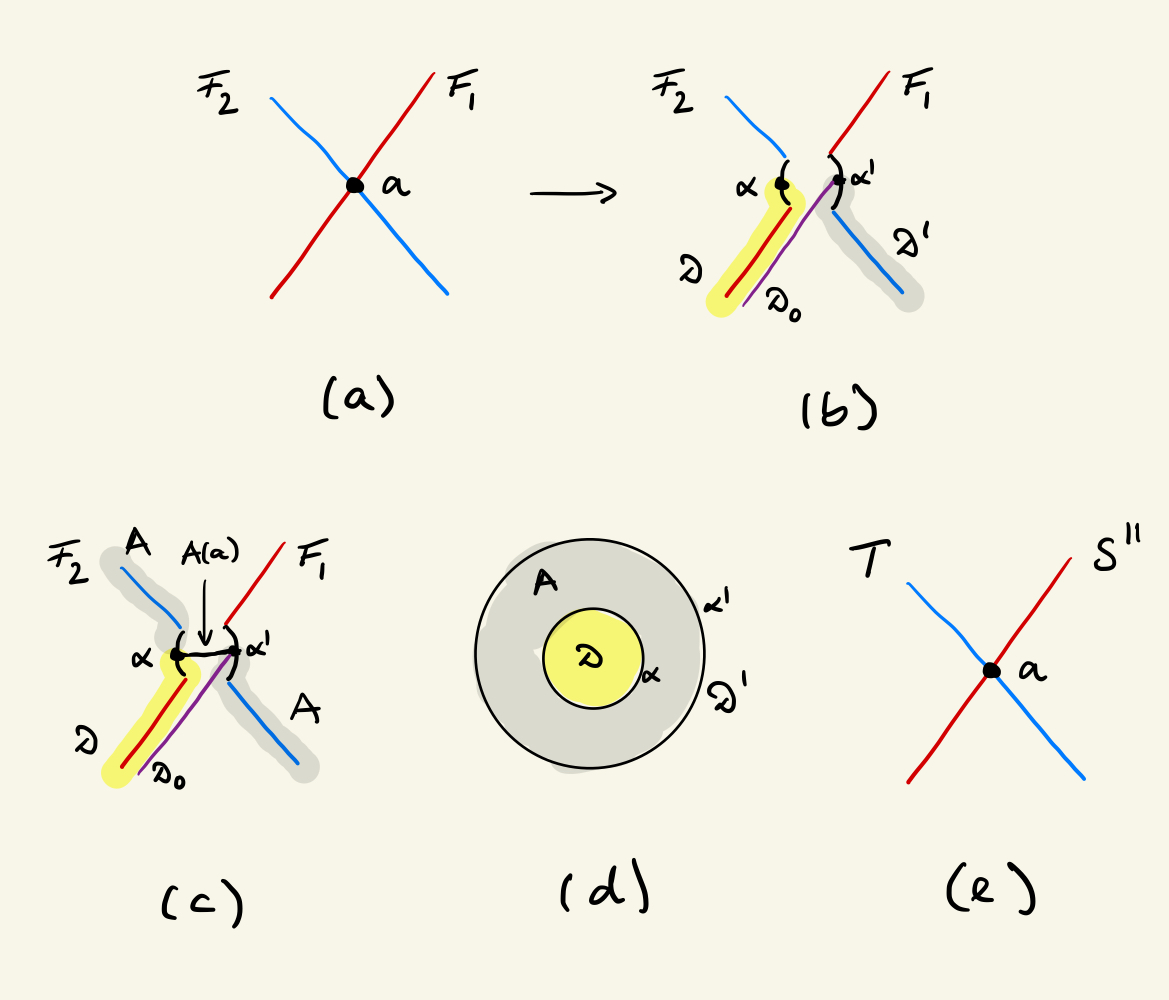}
    \caption{Adjacent discs are not nested}
    \label{fig:Lemma10} 
\end{figure}

\begin{lemma}[Adjacent discs not nested]\label{lem:adjacent_discs_not_nested}
Let $a\subset F_1 \cap F_2$ and denote the associated trace curves $\alpha$ and $\alpha'$. Suppose $\alpha$ bounds a disc $D \subset F_1 + F_2$, $\alpha'$ bounds a disc $D' \subset F_1 + F_2,$ and $D$ and $D'$ are adjacent at $a.$ Then $D' \nsubseteq D$ and $D \nsubseteq D'.$
\end{lemma}

\begin{proof}
By symmetry, it suffices to assume that $D \subset F_1,$ $D' \subset F_2,$ $D \subseteq D';$ see \Cref{fig:Lemma10}(a) and (b). Then $A = D' \setminus D$ is an annulus on $F_1+F_2$ with $\partial A = \alpha \cup \alpha'$; see \Cref{fig:Lemma10}(c) and (d). We also have the exchange annulus $A(a)$ with $\partial A(a) = A(a) \cap A = \partial A.$ Let $S$ denote the component of $F_1+F_2$ containing $D'.$ 

Note that if regular exchanges are performed at all components of $F_1\cap F_2$ except for $a,$ then one obtains $(F + \Sigma) \setminus S$ together with an immersed surface $S'$ having the single curve of self-intersection $a.$ Moreover, $S'$ is the intersection of two embedded normal surfaces $T$ and $S''$, where $T$ is a torus normally isotopic with $A(a) \cap A$ and $S''$ is a surface homeomorphic with $S$, see \Cref{fig:Lemma10}(e).

The trace curve $\alpha'$ bounds a disc $D_0$ in $M$ such that $D_0 \cap (F+\Sigma) = D_0 \cap A(a) = \alpha'.$ The disc $D_0$ is a parallel copy of $D;$ see \Cref{fig:Lemma10}(b).

Since $M$ is irreducible, the 2--sphere $D_0 \cup D' = D_0 \cup A \cup D$ bounds a 3--ball $B$ in $M.$ 

Suppose $S$ is not a 2--sphere. Then $S$ is either a component of the essential surface $F$ or a component of $\Sigma$ linking an ideal vertex. In either case, $S$ is a closed orientable incompressible surface of positive genus. Hence $S\setminus D'$ is not contained in $B.$ We can therefore isotope $D$ to $D_0$ across the product structure of $A$ such that it extends to an isotopy of $S.$ It follows that $S''$ is isotopic with $S$ and $\wt(S) = \wt(S'') + \wt(T) >  \wt(S'').$ If $S$ is a component of $F,$ we obtain a contradiction to the assumption that $F$ is of least weight, since in this case no component of $F$ is contained in $B$, and hence the isotopy extends to an isotopy of $F$ that fixes all components except for $S.$ Hence $S$ is a component of $\Sigma.$ However, each component of $\Sigma$ is separating. Hence $S''$ is separating. The annulus $A$ is disjoint from $S''$ and its boundary curves $\alpha$ and $\alpha'$ are on different sides of $S''$. This is impossible. 

It follows that $S$ is a 2--sphere. Hence $S$ only contains normal triangles. This implies that the three normal surfaces in the sum $S = T + S''$ are only made up of normal triangles. Hence the torus $T$ is a vertex linking surface and therefore incompressible. But $T$ is normally isotopic with $A(a) \cap A$ and this has $D_0$ as a compression disc. This is impossible.
\end{proof}

\begin{lemma}\label{lem:no_disc_patches_on_F}
No patch of $ F_1 + F_2$ on $F$ is a disc.
\end{lemma}

\begin{proof}
We first adapt the argument from \cite[Lemma 2.1]{jaco84-haken}, also borrowing observations from \cite[Lemma 3]{tollefson98-quadspace}.

\begin{quote}
Suppose that a patch $P$ of $F_1 + F_2$ is a disc and contained on the component $S$ of $F.$ Choose notation so that $P \subset F_1$ and denote $a$ the component of $F_1 \cap F_2$ that determines $P.$ The curve $a$ contributes two trace curves in $F + \Sigma,$ say $\alpha$ (bounding $P$) and $\alpha'.$

The trace curve $\alpha'$ is in $F+\Sigma$ and bounds a disc $D_0$ in $M$ such that $D_0 \cap (F+\Sigma)= \alpha'.$ The disc $D_0$ is a parallel copy of $P;$ see \Cref{fig:JO-D0}. Hence $\alpha'$ bounds a disc $D'$ in $F+\Sigma.$ 
\end{quote}

\begin{figure}[h]
    \centering
    \includegraphics[width=7cm]{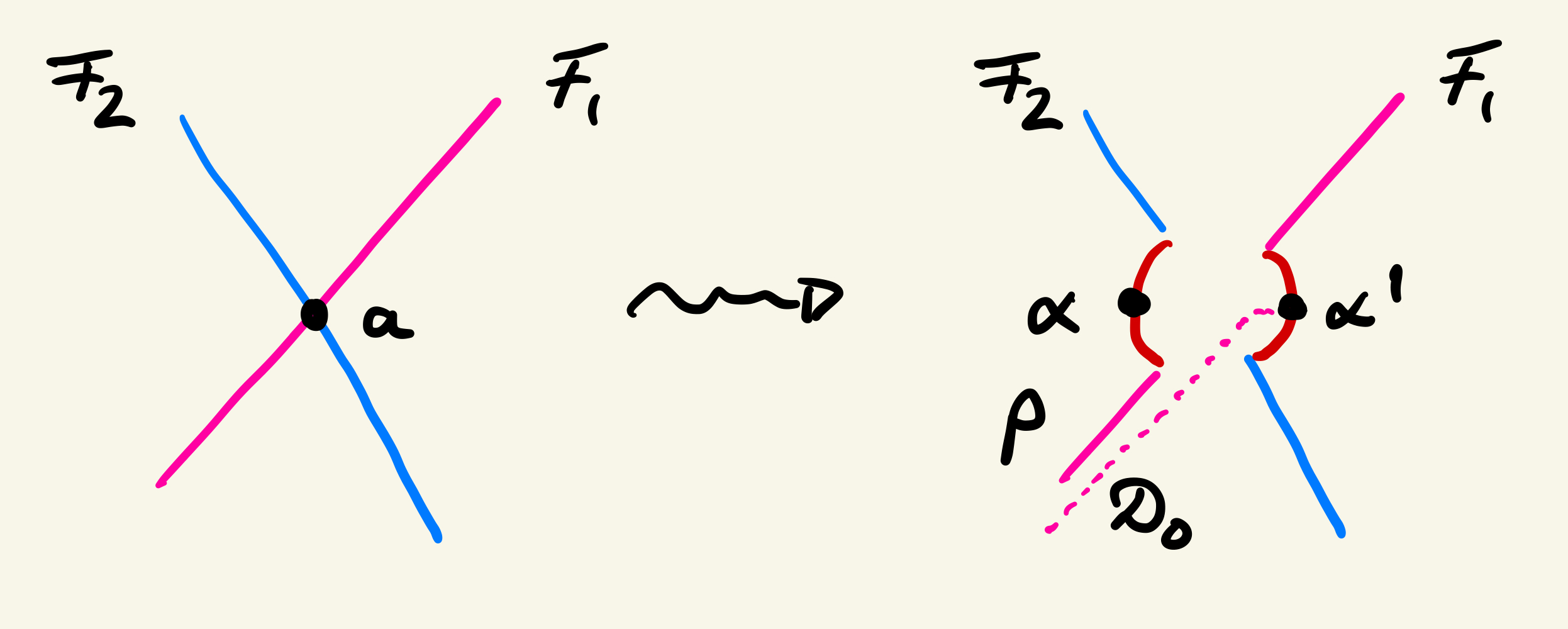}
    \caption{A disc parallel to a patch}
    \label{fig:JO-D0}
\end{figure}

\begin{claim}\label{claim_sphere_switch}
We may choose $D'$ such that $D'$ is adjacent to $P$ at $a.$\footnote{The corresponding claim in \cite[Lemma 2.1]{jaco84-haken} is that if an irregular switch is made at $a,$ then $D'$ and $P$ bound a 2--sphere. The claims are equivalent via \Cref{lem:adjacent_discs_not_nested}.} Moreover, $D'$ is not a patch and it is also contained in $F.$
\end{claim}

\textbf{Proof of claim } 
Suppose the first sentence in the claim is not true.  Then the patch $P'$ adjacent to $P$ at $a$ is not contained in a disc on $F+\Sigma,$ and hence $P'$ lies on a component $S'$ (possibly equal to $S$) of $F+\Sigma$ that has positive genus. Note that $D'$ is a disc on $S'.$ See \Cref{fig:Claim12}(b).

\begin{figure}[h]
    \centering
    \includegraphics[width=11cm]{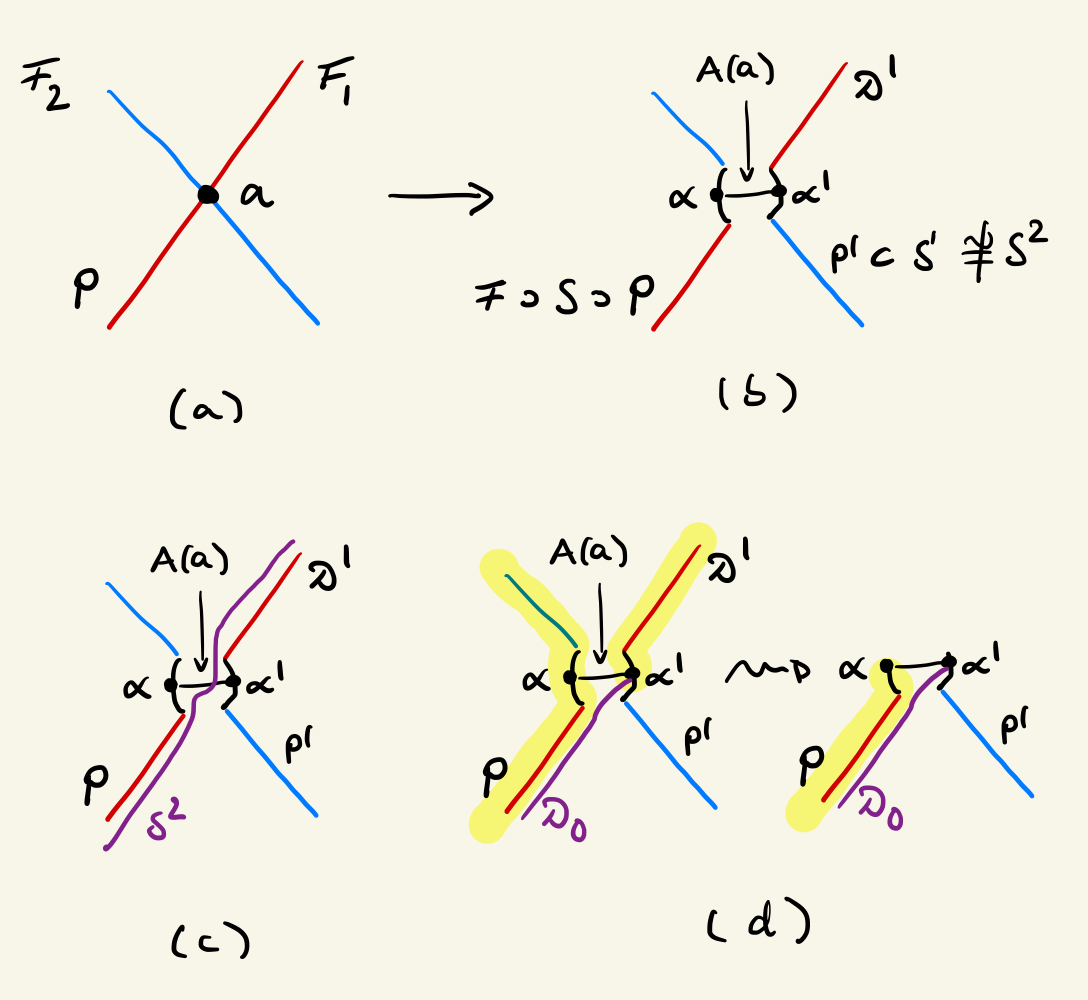}
    \caption{Adjacent disc}
    \label{fig:Claim12} 
\end{figure}

First suppose that $P$ is not a patch in $D'.$ Then the union of $P$, $D'$ and the exchange annulus $A(a)$ is a 2--sphere that can be isotoped to be disjoint from $F+\Sigma.$ The surface $S$ is on one side of this 2--sphere and the surface $S'$ on the other. Since $M$ is irreducible, this implies that a component of $F+\Sigma$ of positive genus would be contained in a 3--ball $B$; see \Cref{fig:Claim12}(c). This is impossible since all components of $F+\Sigma$ of positive genus are incompressible.

Next suppose that $P$ is a patch in $D'.$ In particular, $S = S'$ is a component of $F.$ We obtain a 2--sphere from $D' \cup D_0.$ Since $M$ is irreducible, this 2--sphere bounds a ball $B'$ in $M.$ Since $S$ is incompressible, it is not contained in the interior of $B'.$ As in the previous lemma, we obtain a contradiction to the least weight property of $F$; see \Cref{fig:Claim12}(d).

It follows that we may choose $D'$ to be adjacent to $P.$ 

\begin{quote}
Assume that the disc $D'$ is itself a patch of $F_1 + F_2.$ Then $D' \subset F_2$ and we merely switch the discs $P$ and $D'$, i.e. we make a regular switch only at $a.$ This gives surfaces $F'_i$ isotopic with $F_i$ with $F+\Sigma = F'_1 + F'_2$ and the number of components of $F'_1 \cap F'_2$ is one less than the number of components of $F_1 \cap F_2.$ This is a contradiction to our assumption that $F + \Sigma = F_1 + F_2$ is in reduced form.
\end{quote}

It remains to show that $D' \subset F.$ Hence assume $D' \subset S_1 \subset \Sigma.$ Consider the 2--sphere obtained from $P \cup A(a) \cup D'.$ This bounds a ball $B''$ in $M$ and its interior contains only components of $\Sigma$ linking material vertices. Since $S_1$ is a vertex linking surface, to one side of $D'$ we have the ball $B$ and the component $S$ of $F$, hence to the other side, we have a regular neighbourhood $N(v)$ of a (material or ideal) vertex $v$ with $\partial N(v) = S_1.$ In particular, $(F + \Sigma ) \cap N(v)$ is a finite union of vertex linking surfaces $S_1, S_2, \ldots, S_k.$ See  \Cref{fig:Claim12nested}.

\begin{figure}[h]
    \centering
    \includegraphics[width=8cm]{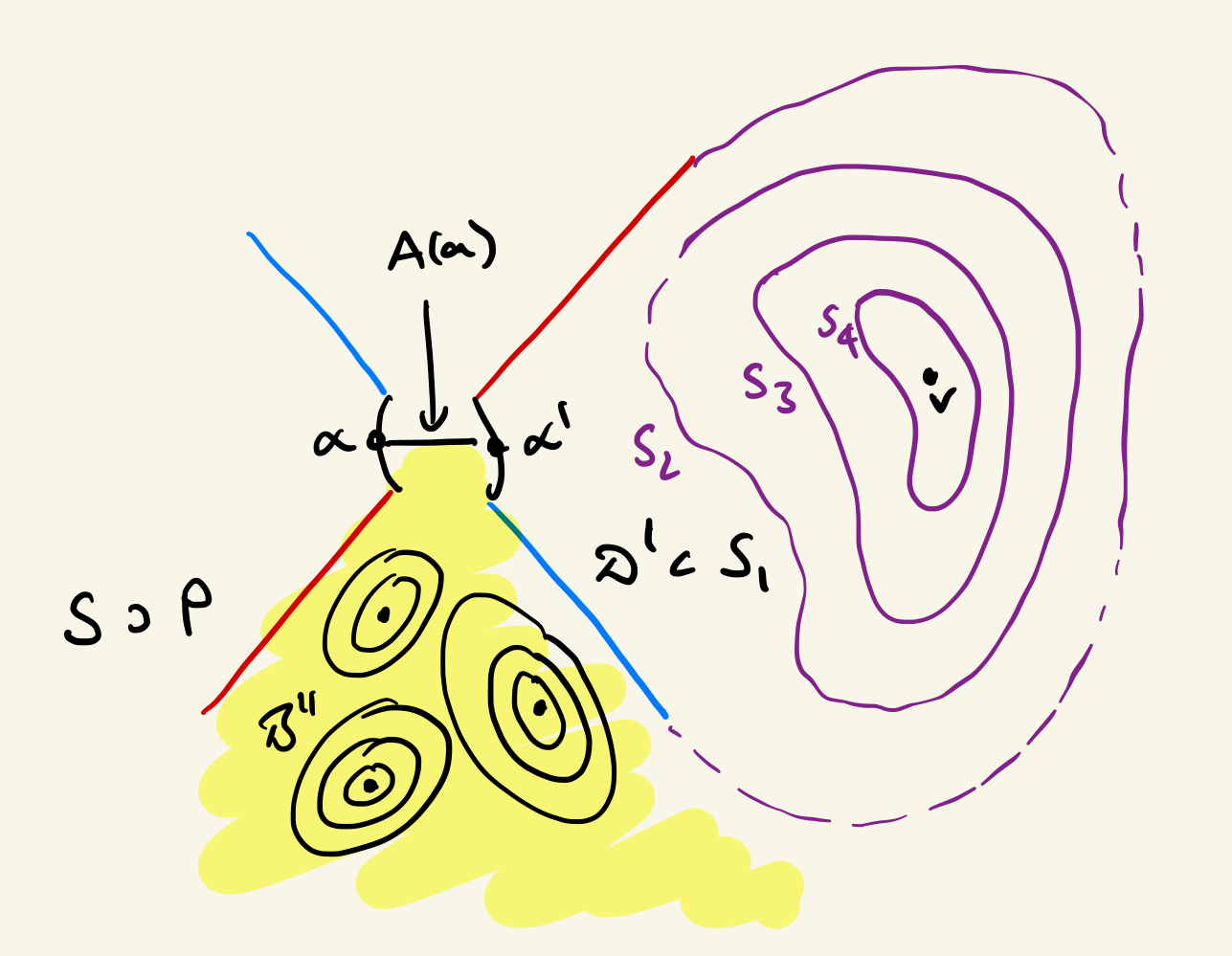}
    \caption{Vertex linking surfaces}
    \label{fig:Claim12nested} 
\end{figure}

Since $D'$ is not a patch it contains an innermost disc patch $P_1\subset D' \subset S_1.$ Its boundary curve is denoted $\alpha_1,$ the corresponding intersection curve $a_1 \subset F_1 \cap F_2$ and the other trace curve $\alpha'_1.$ First note that if $\alpha'_1\subset N(v),$ then $\alpha'_1\subset S_2.$ This follows since $A(a)$ is a 0--weight annulus that is contained in $N(v).$ Consider its non-empty intersection with one tetrahedron. Since $A(a)$ runs between different normal discs, these must be parallel normal triangles since otherwise $A(a)$ would not be contained in $N(v).$ Hence they cannot both be contained in $S_1$ and since $A(a)$ meets $\Sigma$ only in its boundary, it has one boundary component on $S_1$ and one on $S_2.$ Now $\alpha'_1$ bounds a disc on $S_2$. We claim that we may assume that this disc is adjacent to $P_1$. If it is not adjacent, then (since $\alpha_1$ and $\alpha'_1$ are curves on parallel normal triangles), it follows that the link of $v$ must be a sphere. Hence we may choose an adjacent disc on $S_2.$ This cannot be a patch, since otherwise the Haken sum is not reduced. Hence the procedure iterates and we find an innermost disc patch in this disc. Now the ball that we get as above does not contain any other surfaces in $\sigma,$ hence needs to connect to $S_3.$ But this needs to continue indefinitely, contradicting the fact that we only have finitely many surfaces in vertex link.

Hence we need to connect to one of the surfaces in the ball $B''$. Note that there are only finitely many vertices in $B''$ and each only has finitely many vertex linking surfaces. As above, the exchange annulus with boundary on an innermost disc patch either connects to a normally isotopic vertex linking surface (and from then onwards one proceeds towards the associated vertex), or one connects to a surface linking a different vertex. In this case, one constructs a ball $B'''$ containing at least one vertex fewer than $B''$ and the argument repeats; see \Cref{fig:Claim12induction} for a cartoon of the induction.
\begin{figure}[h]
    \centering
    \includegraphics[width=11cm]{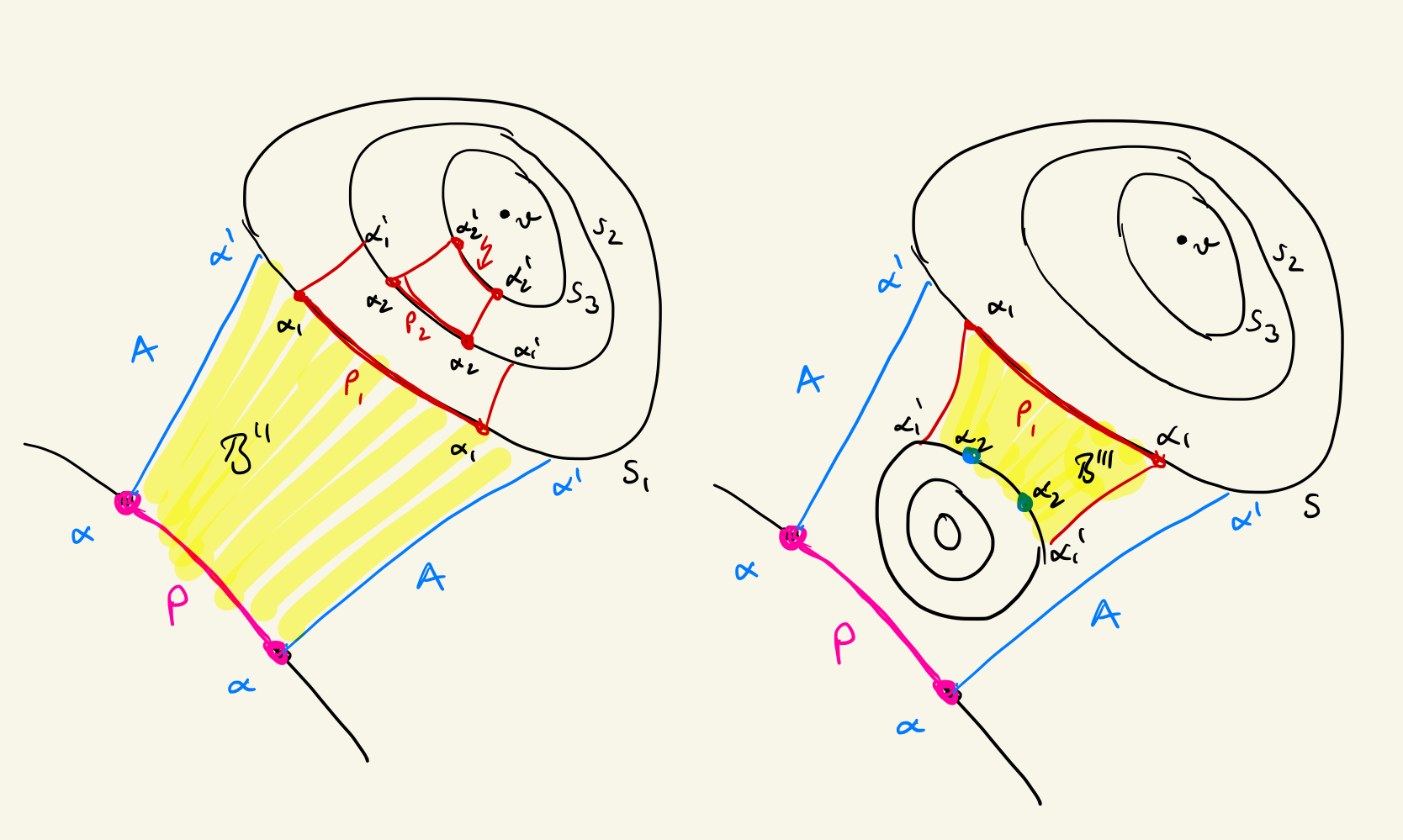}
    \caption{A different perspective on the situation in \Cref{fig:Lemma10}}
    \label{fig:Claim12induction} 
\end{figure}
Since the triangulation has only finitely many vertices and each vertex link only finitely many vertex linking surfaces, one eventually obtains a pair of adjacent disc patches; a contradiction to reduced form.
This proves the claim.
\qed

The proof of the lemma is now completed using a nesting argument. \Cref{claim_sphere_switch} shows that we may assume for any patch $P$ on $F$, which is a disc, that $D',$ the adjacent disc associated with $P$, is not a patch and also contained on $F.$ 

\begin{quote}
Now let $P_0\subset F$ be a patch of $F_1+F_2$ that is a disc and let $P'_0\subset F$ be the disc associated with $P_0$ as above. Since $P'_0$ is not itself a patch, $P'_0$ contains disc patches in its interior. 

We define a sequence of disc patches $P_0, \ldots , P_{n-1}$ inductively: If $P_0, \ldots , P_j$ have the property that none of $P_0, \ldots , P_j$ are patches of $P'_j$, we choose a disc patch $P_{j+1}$ in $P'_j.$ If one of $P_0, \ldots , P_j$ is a patch in $P'_j$, we stop and obtain (after possibly changing notation) a sequence of disc patches $P_0, \ldots , P_{n-1}$ with $P_i$ a patch in $P'_j$ if and only if $j = i-1$ ($1\le i \le n$ reduced modulo $n$).

Let $\alpha_i$ and $\alpha'_i$ be the trace curves that bound the discs $P_i$ and $P'_i$, respectively; and let $a_i$ be the component of $F_1 \cap F_2$ corresponding to $\alpha_i$ and $\alpha'_i,$ $0 \le i \le n.$ Make regular switches in all components of $F_1 \cap F_2$ except $\{a_0, \ldots, a_{n-1}\}.$ 
\end{quote}
From the resulting surface, discard all components that are vertex linking. Recall that all trace curves associated with $\{a_0, \ldots, a_{n-1}\}$ are contained on components of $F.$ Hence we discard $\Sigma$ from the resulting surface.
\begin{quote}
Now at $\{a_0, \ldots, a_{n-1}\}$ instead of making a regular switch, remove the disc $P'_i$ and replace it by a copy of $P_i.$ We get a new normal surface $F'$ in $\tri,$ $F'$ is isotopic in $M$ to $F.$ However, it is possible to connect the annuli $P'_i \setminus P_i$, $1\le i \le n$ (reduced modulo $n$), to get a normal surface $T$, which is a torus or a Klein bottle,\footnote{Indeed, it is a torus as we assumed that all intersection curves are 2--sided.} and $F = F' + T.$ (Notice that $F' \cap T = \{a_0, \ldots, a_{n-1}\},$ see \Cref{fig:JO-telescope}.) However, $\wt(F) = \wt(F') + \wt(T)$ and $\wt(T)\neq 0.$ Since $F'$ is isotopic to $F$ this contradicts the fact that $F$ is of least weight.\newline
Since all possibilities lead to a contradiction, we conclude that no patch of $F + \Sigma$ is a disc.
\end{quote}
\begin{figure}[h]
    \centering
    \includegraphics[width=11cm]{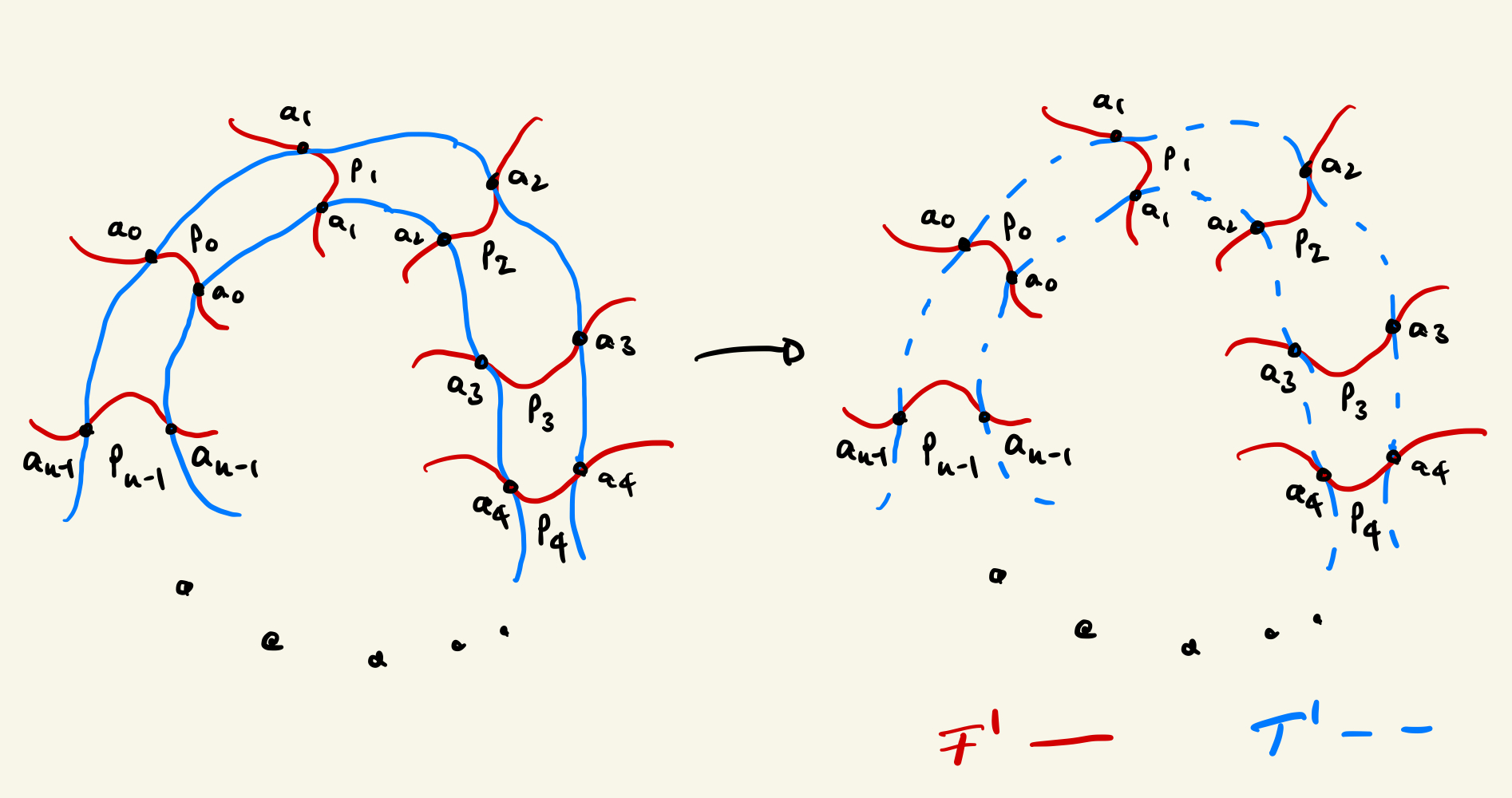}
    \caption{A torus appears}
    \label{fig:JO-telescope} 
\end{figure}
This completes the proof of the lemma.
\end{proof}

We list a number of consequences of \Cref{lem:no_disc_patches_on_F}.
%
Write $\Sigma = \Sigma_m \cup \Sigma_i,$ where $\Sigma_m$ contains all surfaces linking material vertices (that is, all spheres in $\Sigma$).

\begin{corollary}\label{lem:trace_curve_on_sphere}
Let $a \subset F_1 \cap F_2$ and $\alpha$ and $\alpha'$ be the associated trace curves. If $\alpha$ is on $F$, then $\alpha'$ is not on a component of $\Sigma_m.$ \end{corollary}

\begin{proof}
Suppose $\alpha\subset F$ and $\alpha' \subset S \subset \Sigma_m.$ Then $\alpha'$ bounds a disc $D'$ on $S$ from which we can construct a disc $D_0$ in $M$ with $D_0 \cap (F+\Sigma) = \alpha.$ Since $F$ is incompressible, it follows that $\alpha$ bounds a disc $D$ in $F.$ Since the boundary of $D$ is contained in a patch, it follows that $D$ contains a patch that is a disc. This contradicts \Cref{lem:no_disc_patches_on_F}.
\end{proof}

\begin{lemma}\label{lem:trace_curve_on_vertex_links}
Let $a \subset F_1 \cap F_2$ and $\alpha$ and $\alpha'$ be the associated trace curves. If $\alpha$ is on a vertex linking surface $S$, then $\alpha'$ is either on a vertex linking surface normally isotopic with (but not equal to) $S$ or it is on $F.$
\end{lemma}
\begin{figure}[h]
    \centering
    \includegraphics[width=11cm]{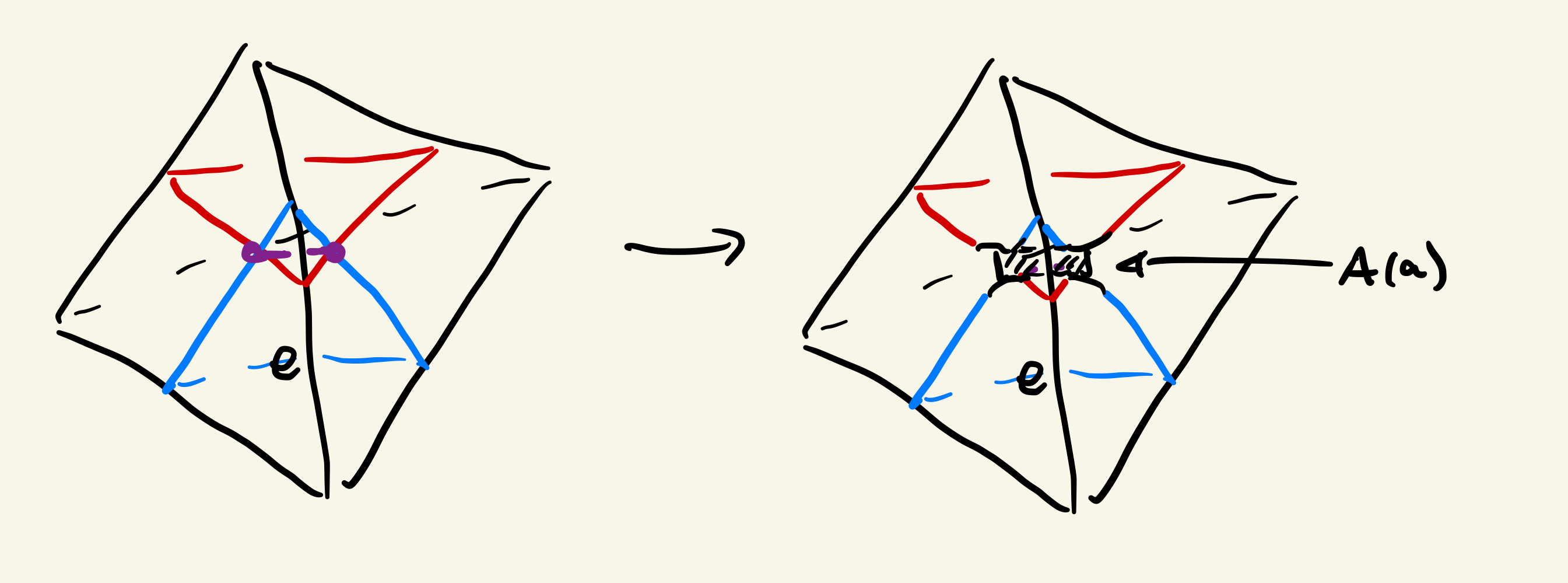}
    \caption{$0$--weight annulus that only spans triangles}
    \label{fig:Lemma14} 
\end{figure}
\begin{proof}
Suppose $\alpha'$ is on a vertex linking surface $S'$ (possibly equal to $S$). Suppose $S'$ is either equal to $S$ or not normally isotopic with $S$. Then $A(a)$ runs between triangles in a tetrahedron that are not normally isotopic.
Since $A(a)$ has boundary only on triangles, propagating around an edge shows that $A(a)$ links an edge $e$ in the triangulation (it may meet a tetrahedron more than once) because every triangle disc is uniquely determined by its intersection with any face of a tetrahedron. See \Cref{fig:Lemma14}. The boundary curves of $A(a)$ bound two adjacent discs $D$ and $D'$ of weight one, attained at $e$. If these are patches, then we have a contradiction to reduced form. The only other possibility is that there are trace curves nested in at least one of these discs, say $D$. Since $A(a)$ meets $F+\Sigma$ only in its boundary curves, any such innermost trace curve is the boundary of an exchange annulus with one boundary curve on $S$ and the other must be on a parallel component to $S.$ We now again obtain a contradiction to reduced form by induction as in the penultimate paragraph in the proof of \Cref{claim_sphere_switch}.
\end{proof}

\begin{corollary}\label{cor:no_material_links}
$\Sigma_m = \emptyset$ 
\end{corollary}

\begin{proof}
We have $F + \Sigma = F_1 + F_2.$ 
Since for each component of $F_1\cap F_2$ the two associated trace curves are either both on $F+\Sigma_i$ or both on $\Sigma_m$, it follows that we can write $F_i = F'_i \cup F''_i$ such that $F+\Sigma_i= F'_1 + F'_2$ and $\Sigma_m= F''_1 + F''_2.$ Hence all components of $F''_1$ and $F''_2$ are only made up of normal triangles and hence are vertex linking surfaces. Since $F_i$ has no vertex linking components, this implies $F''_i= \emptyset.$ Hence $\Sigma_m=\emptyset.$
\end{proof}

\begin{corollary}\label{cor:no_disc_patches}
No patch of $F_1 + F_2$  is a disc.
\end{corollary}

\begin{proof}
It remains to show that no patch of $F_1 + F_2$ contained in $\Sigma=\Sigma_i$ is a disc because 
\Cref{lem:no_disc_patches_on_F} shows that no patch contained in $F$ is a disc and \Cref{cor:no_material_links} shows that $\Sigma_m = \emptyset$.

Suppose a patch $P$ on $\Sigma_i$ is a disc, contained on the surface $S$ linking $v.$ Let $\alpha = \partial P$ and note that the corresponding trace curve $\alpha'$ bounds an embedded disc and hence cannot be on $F.$ Hence is also on $\Sigma.$ Then \Cref{lem:trace_curve_on_vertex_links} implies that it is on a normally isotopic but different surface $S'$. We again obtain a contradiction via an induction as in the penultimate paragraph in the proof of \Cref{claim_sphere_switch}.

There are two cases to consider. If $S'$ is outside of the vertex neighbourhood $N(v)$ bounded by $S,$ then, since $S$ is not a sphere and no patch on $F$ is a disc, there must be infinitely many surfaces normally isotopic to $S$ outside of $N(v).$ If $S'$ is inside $N(v)$ we again obtain an infinite sequence of vertex linking surfaces in $N(v).$ Both scenarios are impossible.
\end{proof}

\begin{lemma}\label{lem:incompressible_patches_ideal}
Each patch of $F_1 + F_2$ is incompressible.
\end{lemma}

\begin{proof}
As in the proof of \cite[Lemma 2.1]{jaco84-haken}, we note that this follows from \Cref{cor:no_disc_patches}: If some patch $P$ is compressible, then there is a disc $D \subset M$ with $D \cap P = \partial D \cap P$ and $\partial D$ does not bound a disc in $P.$ Since $F+\Sigma$ is incompressible, we can also assume that $D \cap F = \partial D \cap F.$ Now $\partial D$ bounds a disc $D'$ on $F+\Sigma.$ But then $D'$ contains a patch that is a disc, which is not possible by the above.
\end{proof}

\subsection{Proof of \Cref{thm-mixed:some extremal is closed essential}}
\label{subsec:proof-of-J-O}

Suppose $M$ is the interior of an irreducible and $\partial$--irreducible, compact, orientable 3--manifold with (possibly empty) boundary, and let $\tri$ be a mixed triangulation of $M.$

Suppose $M$ contains a closed, essential surface. Haken's approach to normalising surfaces with respect to triangulations (as described, for instance, in \cite[\S 3.1]{jaco03-0-efficiency} and \cite[Chapter 3]{matveev03-algms}) implies that there is a closed, essential, normal surface $F$ in $M.$ It remains to show that $F$ may be chosen such that $F$ is a $Q_0$--vertex surface. Replace $F$ by a normal surface that has least weight amongst all normal surfaces isotopic (but not necessarily normally isotopic) to $F.$ We let $F$ denote this surface. 

Suppose $F$ is not a $Q_0$--vertex surface. According to \Cref{lem:sum_of_verices}, we can write $$n x(F) = \sum n_j x(V_j),$$ where $n, n_j \in \mathbb{N}$ and each $V_j$ is a $Q_0$--vertex surface. 
Hence we may write
\begin{equation}\label{eq:HakenSumExtremal}
nF + \Sigma = V + W
\end{equation}
where $\Sigma$ is a finite sum of pairwise disjoint vertex linking surfaces disjoint from $nF$, $V = V_i$ and 
\[x(W) =  -x(V_i) + \sum_{j\neq i} n_j x(V_j)\]
Moreover, \Cref{cor:no_material_links} implies that each surface in $\Sigma$ has positive genus. 
Now $nF$ is 2--sided and of least weight and each intersection curve in $V \cap W$ is 2--sided. We now adapt the proof of \cite[Theorem 2.2]{jaco84-haken} to show that $V$ is incompressible. To avoid confusion, we call the surface $V$ in \Cref{eq:HakenSumExtremal} the \textbf{extremal surface summand}.

Write $nF + \Sigma = V' + W'$ in reduced form with $V'$ isotopic with $V$ and $W'$ isotopic with $W$ in $M.$ We know that no patch of $V' + W'$ is a disc (\Cref{cor:no_disc_patches}), each patch is incompressible (\Cref{lem:incompressible_patches_ideal}) and $\Sigma$ consist only of surfaces linking ideal vertices (\Cref{cor:no_material_links}). In particular, $nF + \Sigma$ is incompressible.

\textbf{The extremal surface summand is not a sphere.}
Suppose first that $V$ is a 2--sphere. Then $V'$ is a 2--sphere. If $V' \cap W'\neq \emptyset,$ then the curves in $V' \cap W'$ divide $V'$ into patches. Since $V'$ is a 2--sphere, at least one of these patches is a disc. But this contradicts \Cref{cor:no_disc_patches}: no patch is a disc if $nF + \Sigma = V' + W'$ is in reduced form. Hence $V' \cap W'=\emptyset.$ Then $V'$ is a component of $nF + \Sigma.$ Since $V'$ contains at least one quadrilateral disc, it is not a component of $\Sigma.$ Hence $V'$ is a component of $nF.$ But $nF$ is an essential surface, so in particular $V'$ is not a 2--sphere. This contradicts our assumption that $V$ is a 2--sphere. Hence $V$ is not a 2--sphere.

Since $V = V_i$ is an arbitrary summand from $n x(F) = \sum n_j x(V_j),$ it follows that none of the vertex surfaces $V_j$ is a 2--sphere. Hence $\chi (V_j) \le 0.$ This, together with the fact that Euler characteristic is additive over Haken sums and the fact from \Cref{cor:no_material_links} that no surface in $\Sigma$ is a sphere implies the statement concerning the Euler characteristic. It remains to show that $V_i$ is essential. We first show that $V_i$ is incompressible and then that it is not boundary parallel.

\textbf{The extremal surface summand is incompressible.}
Since $V'$ is isotopic to $V$, it suffices to show that $V'$ is incompressible. 

\begin{quote}
Assume that $V'$ is compressible. Then there is a disc $D \subset M$ such that $D \cap V' = \partial D$ and $\partial D$ is not contractible in $V'.$

Amongst all discs $D\subset M$ such that $D \cap V' = \partial D$ and $\partial D$ not contractible in $V'$, choose one, still denoted by $D,$ so that $D$ is transverse to $W'$ and $D\cap W'$ has a minimal number of components. 
\end{quote}

As in the proof of \cite[Theorem 2.2]{jaco84-haken} we first show that $nF+\Sigma$ incompressible and $M$ irreducible implies that $D \cap W'$ is a non-empty union of spanning arcs for $D$. All the ideas in the following paragraphs are from \cite{jaco84-haken}, just slightly paraphrased.

Suppose some component of $D\cap W'$ is a simple closed curve. Then there is some simple closed curve $\gamma \subset  D\cap W'$ such $\gamma$ bounds a disc $D'\subset D$ and $D'\cap W' = \gamma.$ Such $\gamma$ is called \textbf{innermost}. Since $\gamma$ is disjoint from $V'$, it is contained in a patch $P$ of $V'+W'.$ Since each patch is incompressible, $\gamma$ bounds a disc $D''\subset P.$ Now $D''$ may contain components of $D\cap W'.$ Each such component is a simple closed curve since it is contained in the interior of $D$ and the intersection is transverse. We now choose an innermost curve $\gamma'$ on $D'',$ i.e. $\gamma'$ bounds a disc $D''' \subset D''$ and $D''' \cap D' = \gamma.$ This implies that $\gamma'$ also bounds an innermost disc $D''''$ on $D$, and hence we obtain a 2--sphere $D''' \cup D''''$ in $M$ with the property that $(D''' \cup D'''') \cap (V' \cup W') = D'''.$ Since $M$ is irreducible, this 2--sphere bounds a ball in $M$, and since each patch is incompressible and no patch is a disc, this ball only meets $V' \cup W'$ in $D'''.$ Hence there is a isotopy of $D$ across a small neighbourhood of this ball that reduces the number of intersections of $D$ with $W'$ by one (namely $\gamma'$). This contradicts the choice of $D.$ Hence $D \cap W'$ only contains spanning arcs for $D,$ and hence $D$ has the least number of spanning arcs amongst all such compression discs for $V'$ that are transverse to $W'.$

Now the spanning arcs cut $D$ into a number of regions. A region $\Delta$ is a disc in $M$ and its boundary consists of arcs that alternate between $V'$ and $W'$. Each endpoint $p$ of such an arc lies on some curve $a$ of intersection in $V' \cap W'$. Depending on the regular exchange at $a$, the arcs on $V'$ and $W'$ that meet at $p$ can be connected to give an arc on $nF + \Sigma$ (this is called a \textbf{good} corner of $\Delta$) or they cannot be connected (called a \textbf{bad} corner of $\Delta$). See \Cref{fig:GoodBad}, where \emph{the good} corner is indicated with a $g$ and \emph{the bad} corner with a $b$ \emph{and the ugly} terminology is from \cite{jaco84-haken}.
\begin{figure}[h]
    \centering
    \includegraphics[width=11cm]{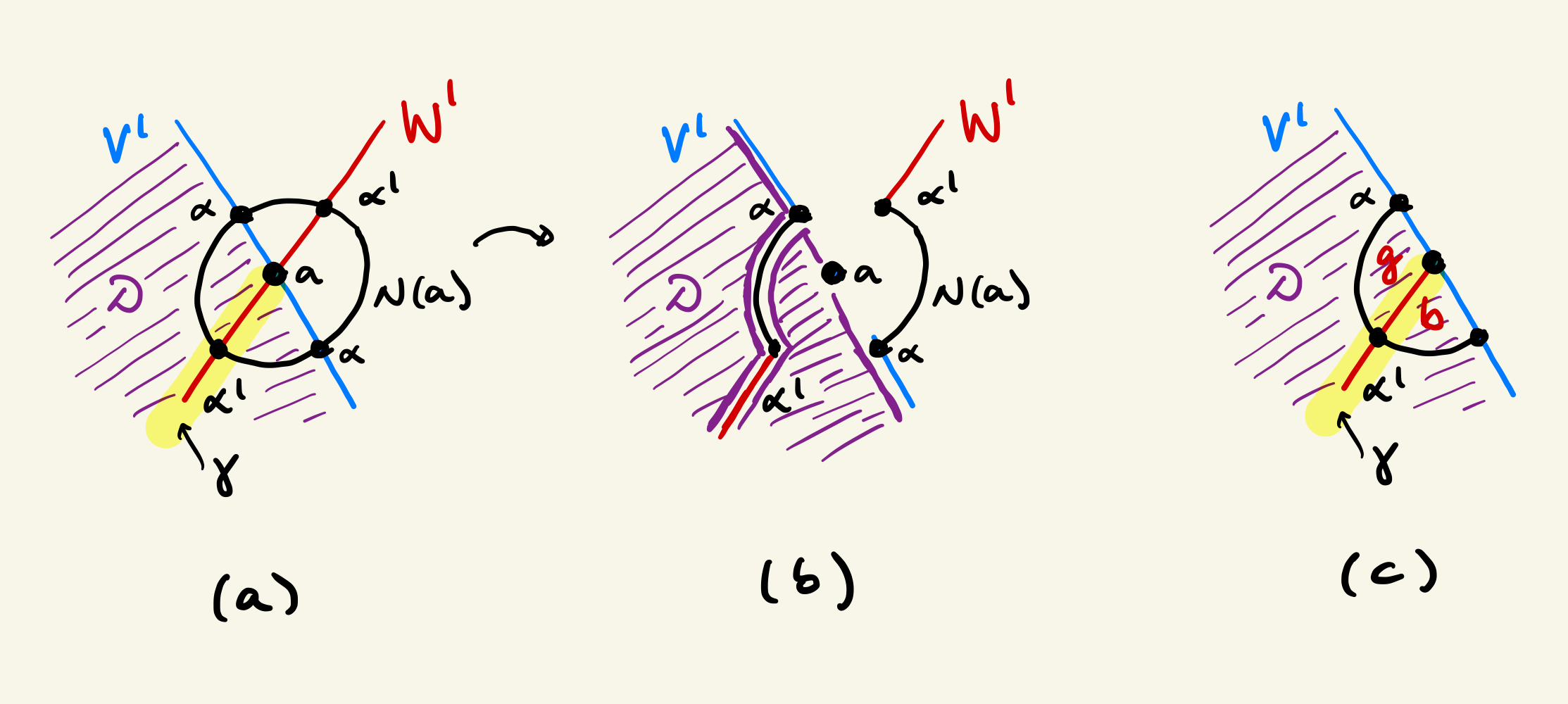}
    \caption{The good, the bad and the ugly}
    \label{fig:GoodBad} 
\end{figure}
The following counting argument credited to Haken shows that there is at least one region with at most one bad corner: $k$ spanning arcs cut the disc into $k+1$ regions. The labelling introduces $2k$ bad corners. Hence at least one region contains fewer than two bad corners.

If $\Delta$ has no bad corner, then $\Delta$ gives a disc (also called $\Delta$) in $M$ with boundary on $nF + \Sigma$ and interior disjoint from $nF + \Sigma.$ Since each component of $nF + \Sigma$ is incompressible, there is a disc $\Delta'$ on $nF + \Sigma$ with boundary equal to $\partial \Delta.$ 
\begin{quote}
Now $\partial \Delta = \partial \Delta'$ can be decomposed into arcs so that the arcs are alternately in $V'$ and $W'$ as one traverses $\partial \Delta'.$ Furthermore, there are parts of trace curves in $\Delta',$ each of which is a spanning arc of $\Delta'$ with its endpoints among the endpoints of the arcs that decompose $\partial \Delta'$ (see \Cref{fig:JOFig5}).
\begin{figure}[h]
    \centering
    \includegraphics[width=10cm]{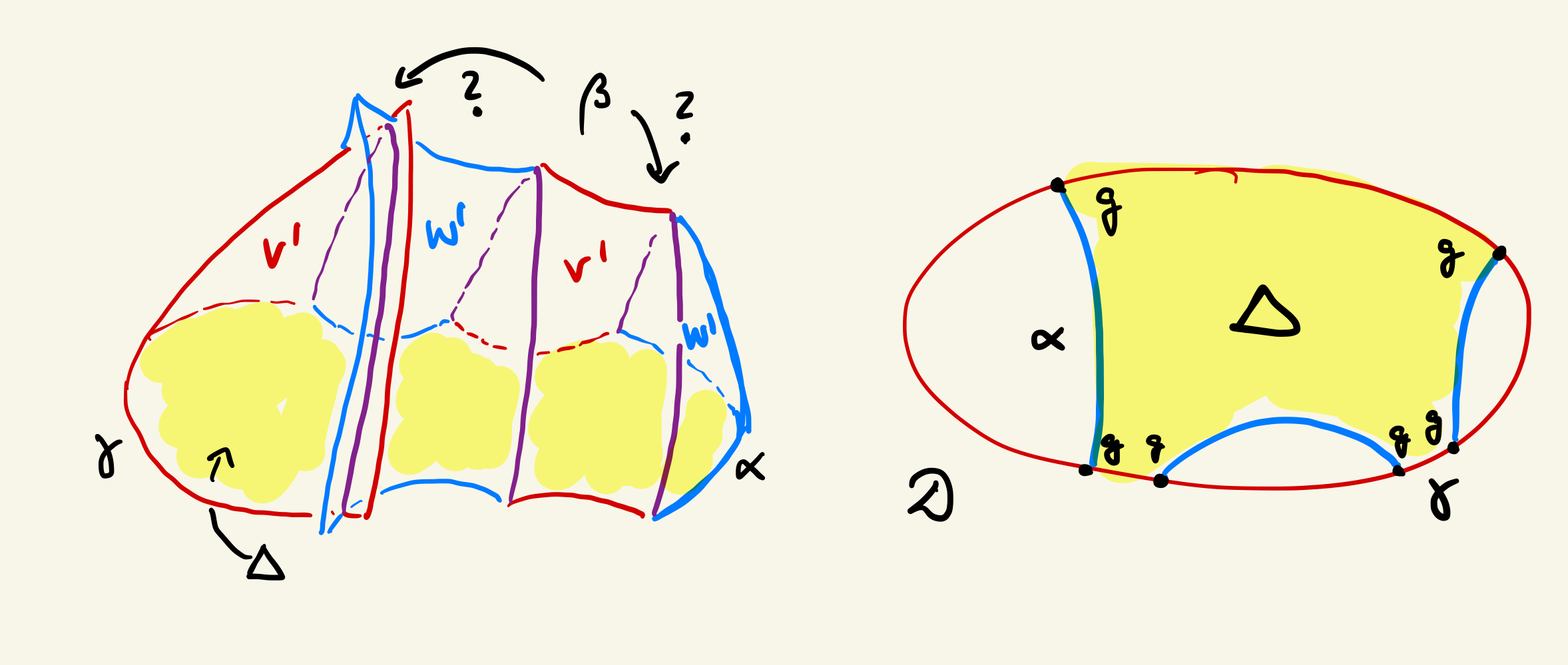}
    \caption{Spanning arcs}
    \label{fig:JOFig5} 
\end{figure}

Let $\beta$ be an outermost arc on $\Delta'.$ That is, $\beta$ is a trace arc which separates $\Delta'$ into discs $E$ and $E'$ so that $E$ meets the collection of trace arcs only in $\beta.$ There are two possibilities for $E.$ Namely $\partial E = \alpha \cup \beta$ where $\alpha \subset W'$ or $\partial E = \gamma \cup \beta$ where $\gamma \subset V'.$

In the case where $\alpha \subset W'$, we consider a boundary compression of $D$ at the disc E. The disc $E$ is a disc in $M$ with $\partial E = \alpha \cup \beta$ where $\alpha$ is an arc in $D$ and $\beta$ is an arc in $W'.$ A boundary compression of $D$ at $E$ gives us two discs $D_1$ and $D_2$ such that $D_i \cap V' = \partial D_i$ and homotopically $\partial D$ is the sum of $\partial D_1$ and $\partial D_2$ in $V'$. By our assumption that $\partial D$ is not contractible in $V'$, one of $\partial D_1$ and $\partial D_2$ is not contractible in $V'.$ However, both $D_1 \cap W'$ and $D_2 \cap W'$ have fewer components than $D \cap W'.$ This contradicts our choice of $D.$

In the case $\partial E = \gamma \cup \beta$ where $\gamma \subset V',$ $E \subset V'$ and we simply slide $\gamma$ across $E$ and past $\beta$ in $V'$ to get an isotopy of $D$ keeping $\partial D$ in $V'$. Furthermore, this isotopy exchanges two components of $D \cap W'$ for one. This is a contradiction to our choice of $D.$
\end{quote}

Now we suppose there is a region $\Delta$ with exactly one bad corner. In this case, one obtains a disc $\Delta$ in $M$ with boundary consisting of two arcs: one arc $\gamma$ contained in a component of $nS + \Sigma$ and one arc $\gamma'$ contained in an annulus in $M$ on the boundary of the neighbourhood of the curve $a \subset V' \cap W'$ containing the bad corner of $\Delta.$ Recall that the trace curves associated with $a$ are contained on incompressible surfaces.
\begin{quote}
Let $N(a)$ be the solid torus neighbourhood about $a$, then $V' \cap B'$ separates $\partial N(a)$ into four annuli (noting that $a$ is orientation preserving on $V'$ and hence on $W'$ also). A regular switch at the curve $a$ puts two of the four annuli in $\partial N(a)$ into $nS + \Sigma$. 

One of the remaining ones meets $\Delta$ in a small spanning arc of $\Delta$ close to the bad corner. So we have an annulus, say $A'$, in $\partial N(a)$ that meets $\Delta.$

In this situation, the disc $\Delta$ determines a disc in $M$, which we shall also call $\Delta$, so that $\partial \Delta$ can be decomposed into two arcs $\delta$ and $\delta'$ with $\delta \subset nS + \Sigma$ and $\delta'$ a nonseparating spanning arc of the annulus $A'$ (see \Cref{fig:JOFig6}).
\begin{figure}[h]
    \centering
    \includegraphics[width=10cm]{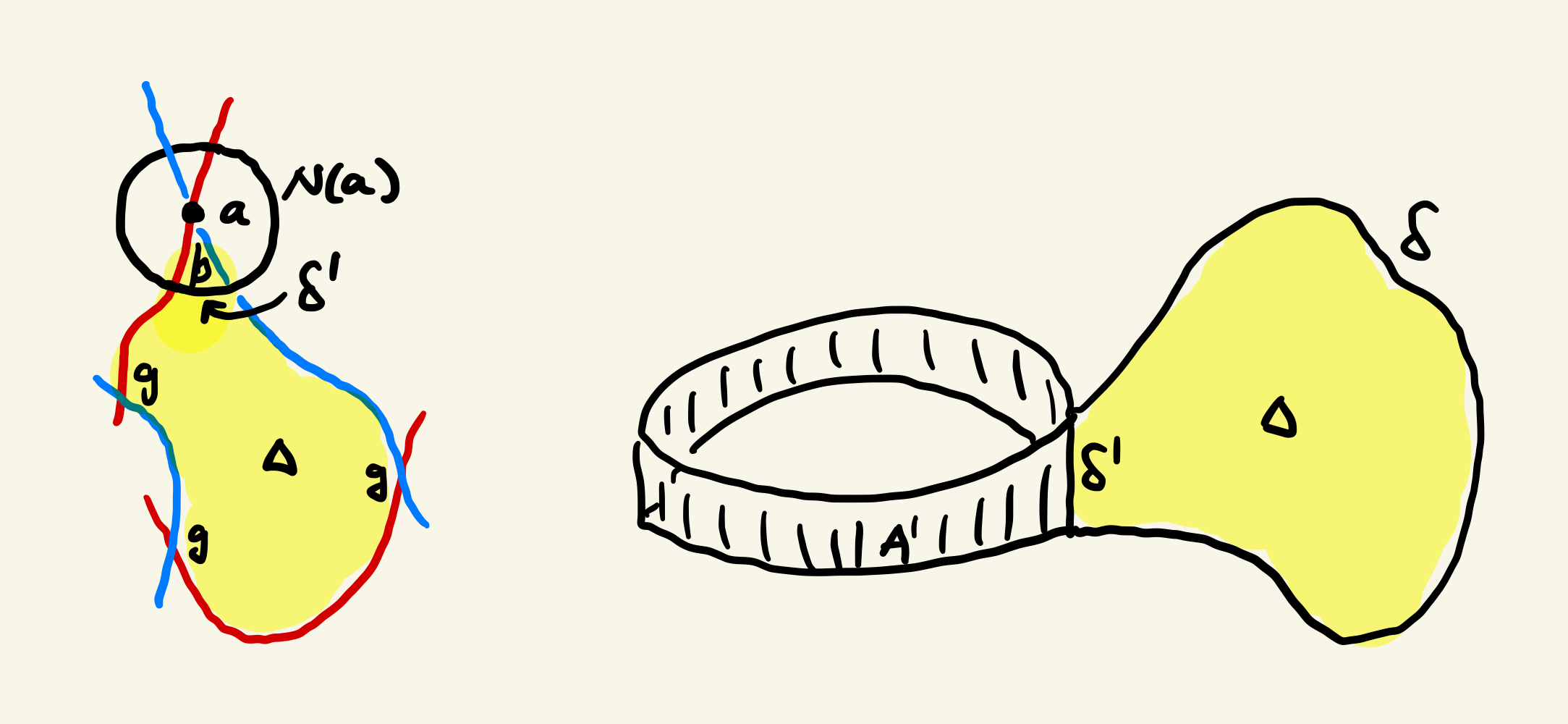}
    \caption{A disc appears}
    \label{fig:JOFig6} 
\end{figure}

Let $\Delta_1$ and $\Delta_2$ be parallel copies of $\Delta$ meeting $A'$ in $\delta'_1$ and $\delta'_2$, parallel to $\delta'$ in $A'$ and meeting $nF + \Sigma$ in $\delta_1$ and $\delta_2$, parallel to $\delta$ in $nF + \Sigma.$ The arcs $\delta'_1$ and $\delta'_2$ divide the annulus $A'$ into two (rectangular) discs $A'_1$ and $A'_2$ where $\Delta \cap A'_2 = \delta'.$ Since the surface $nF + \Sigma$ is assumed to be two-sided, the union of the discs $\Delta_1,$ $\Delta_2$ and $A'_1$ is a disc $\mathcal{D}$ with $\mathcal{D} \cap (nF + \Sigma) = \partial \mathcal{D}.$ so there is a disc $\mathcal{D}' \subset nF+\Sigma$ such that $ \partial \mathcal{D}' =  \partial \mathcal{D}$ (see \Cref{fig:JOFig7}).
\begin{figure}[h]
    \centering
    \includegraphics[width=10cm]{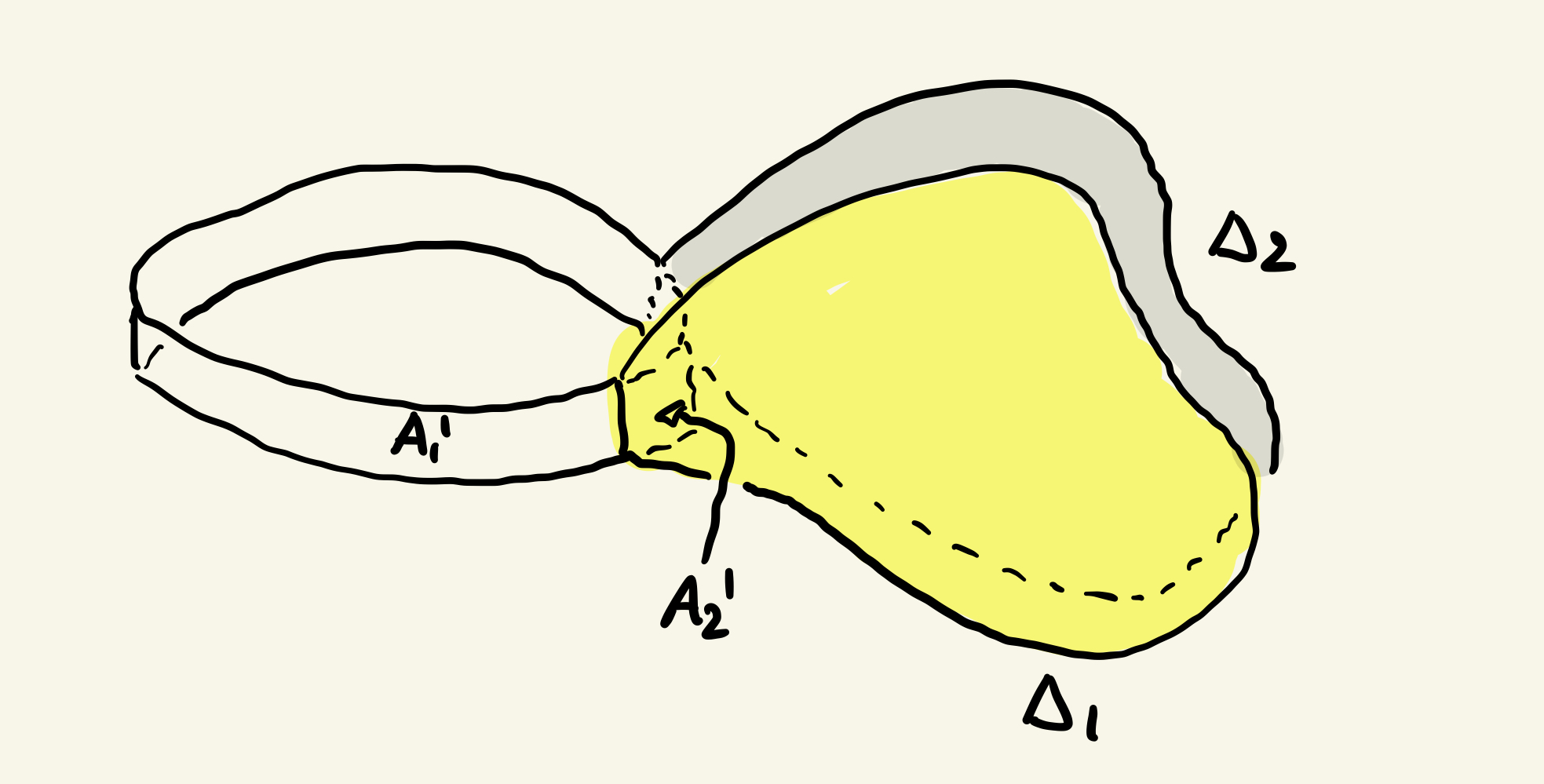}
    \caption{$\mathcal{D} = A'_1 \cup \Delta_1 \cup \Delta_2$}
    \label{fig:JOFig7} 
\end{figure}

First, observe that $\partial A'$ is not contained in $\mathcal{D}'.$ For if this were the case, then since $\partial A'$ is parallel to trace curves in $nF+\Sigma,$ some patch of $V'+W'$ would be a disc. This is a contradiction to \Cref{lem:incompressible_patches_ideal}.

Now the union $\mathcal{D} \cup \mathcal{D}'$ is a 2--sphere in $M$ and so must bound a 3--ball $B$ in $M.$ The disc $\Delta$ is not in the 3--ball $B.$ Otherwise $nF + \Sigma$ would be contained in $B$ and so $F$ could not be incompressible. 

The only possibility is that there is an annulus $A$ in $nF+\Sigma$ with $\partial A = \partial A'$ and the torus $A \cup A'$ bounds a solid torus $T$ with meridian disc $\Delta.$ Since $\Delta$ meets $A'$ in the arc $\delta'$ and $A$ in the arc $\delta$, $T$ is a product between $A$ and $A'.$ So there is an isotopy of $M$ moving $A$ to $A'$ and not moving points outside a small neighbourhood of $T$. Let $F'$ be the image of $nF + \Sigma$ under this isotopy ($F'$ is obtained from $nF + \Sigma$ by exchanging the annulus $A$ in $nF + \Sigma$ for the annnulus $A'$). 
But while $F'$ is not a normal surface with respect to $\tri$ (it contains a fold), it is isotopic in $M$ to a normal surface $F''$ with $\wt (F'') < \wt( nF + \Sigma)$ (see \Cref{fig:JOFig8}).
\begin{figure}[h]
    \centering
    \includegraphics[width=10cm]{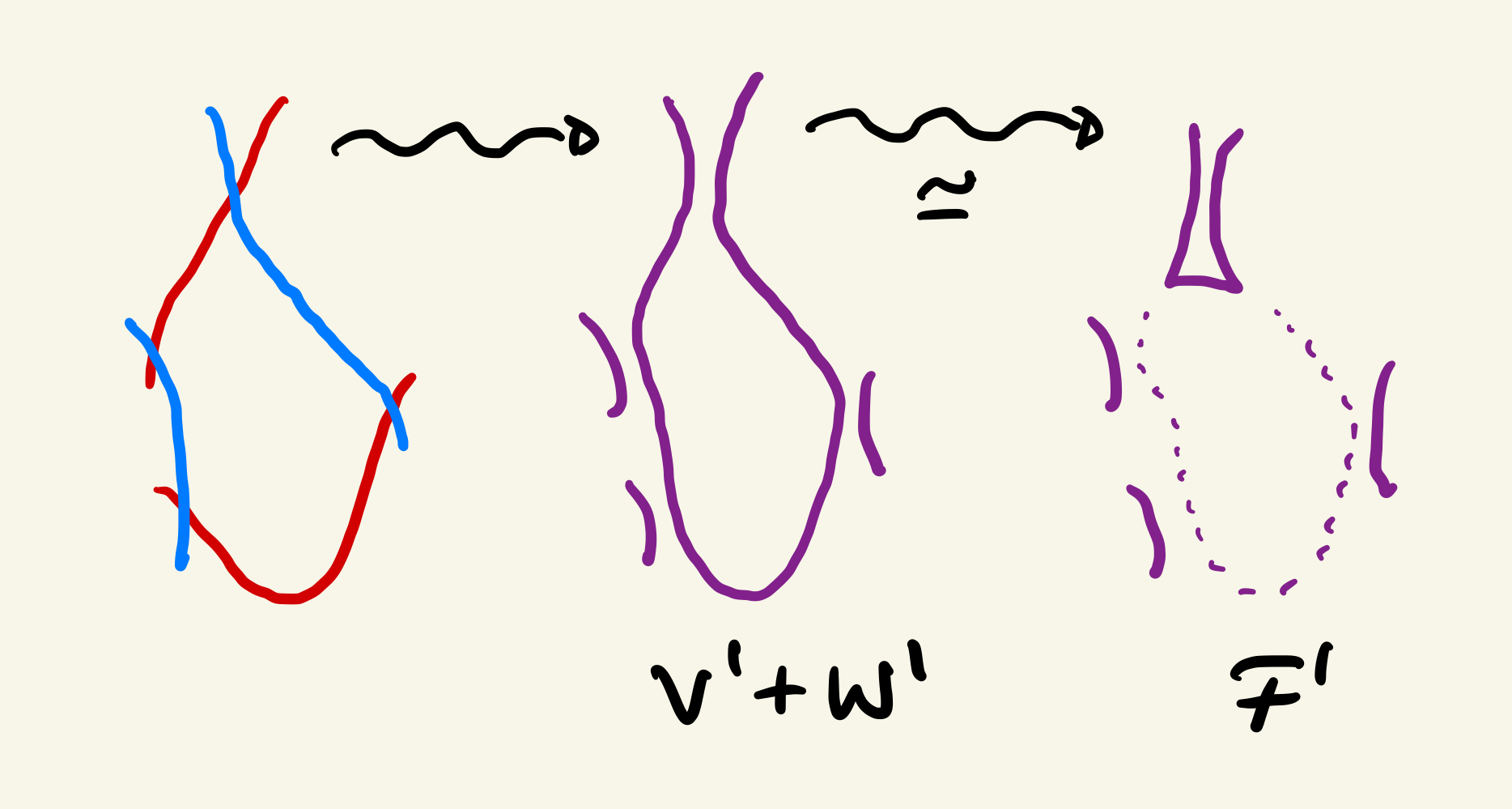}
    \caption{Isotopy across annulus}
    \label{fig:JOFig8} 
\end{figure}
\end{quote}
Do deal with vertex linking surfaces (which we do not know to be of least weight), we need to analyse this situation in a little more detail. Let $S_0$ be the component containing $A.$ Then the isotopy takes $S_0$ to a surface $S'_0$ in the complement of $(nF + \Sigma)\setminus S_0.$ Moreover, the components of $(nF + \Sigma)\setminus S_0$ are barriers to the normalisation of $S'_0.$

If $S_0 \subseteq nF,$ then $S'_0$ normalises to a normal surface of lower weight and hence we have a contradiction to the least weight of $nF.$

Hence $S_0 \subseteq \Sigma.$ In this case note that performing regular exchanges at all curves in $V' \cap W'$ except for $a$, which we leave as an intersection curve, results in the normal surface $(nF + \Sigma) \setminus S_0$ together with one connected immersed normal surface $S''_0$ with a single curve of self-intersection $a.$ After performing a regular exchange along $a,$ $S''_0$ is turned into the vertex linking surface $S_0.$ However, this situation is ruled out by \Cref{lem:trace_curve_on_vertex_links}.

\textbf{The extremal surface summand is not boundary parallel.}
It remains to show that $V$ is not isotopic to a vertex linking surface.
Suppose that $V$ (and therefore $V'$) is isotopic to the link of some vertex of genus $\ge 1.$ Let $B_v$ be such a vertex link isotopic to $V'$ and disjoint from the closed surface $W'.$ Then there is a product region $N \cong B_v \times I$ in $M$ with $\partial N = B_v \cup V'.$ 
Suppose that there is a connected component $U$ of $W' \cap N$ with non-empty boundary. Then $U$ is a patch and \Cref{lem:incompressible_patches_ideal} implies that it is incompressible in $N$. Now $U \cap B_v \subseteq W' \cap B_v = \emptyset$ implies that $\partial U \subset V'.$ Hence \cite[Corollary 3.2]{waldhausen68-large} implies that $U$ is parallel to a subsurface $U'$ of $V'.$ If the interior of $U'$ does not contain any other components of $\partial U,$ then we say that $U$ is \textbf{innermost}. Choosing an innermost component $U$ of $W' \cap N$, we see that performing regular exchanges at all intersection curves in $U \cap V'$ gives a contradiction to the fact that $nF + \Sigma = V' + W'$ is reduced.

Hence $W' \cap V'=\emptyset.$  But then a component of $nF + \Sigma$ is isotopic but not normally isotopic to a vertex linking surface. Hence a component of $nF$ is isotopic to a vertex linking surface, giving the final contradiction. This completes the proof of \Cref{thm-mixed:some extremal is closed essential}. 


\subsection{Example}
\label{subsec:example}

The trivial circle bundle over a once-punctured surface of genus two has a triangulation $\tri_M$ with isomorphism signature 
\begin{center}
{\tt sLLLLPLPMvQAQbefijjlklkjpqqoorrraxaaaaaaaaxhaaaahhh}.
\end{center}
 This triangulation was created for us by Mark Bell using  {\tt flipper}~\cite{flipper}.

There are 29 admissible vertex solutions spanning $Q(\tri_M),$ and this set is identical with the set of fundamental solutions. Exactly 9 of the corresponding fundamental surfaces are thin-edge linking surfaces of genus two. Hence after one compression they reduce to a boundary parallel torus. The remaining 20 vertex surfaces are spun-normal. In particular, no fundamental surface is an essential torus. This example highlights that it is necessary to work with $Q_0(\tri_M)$ in general. It also points to a gap in the proof of \cite[Theorem 5.5]{kang05-spun}; namely spun-normal annuli are not considered. The statement of the Theorem \emph{may} be correct.

The space $Q_0(\tri_M)$ has 81 admissible vertex solutions, and these again coincide with the admissible fundamental solutions. The corresponding surfaces are 9 thin edge linking surfaces of genus two, 4 separating essential tori and the remaining 68 are non-separating essential tori.

\section{Algorithms}
\label{sec:Detecting compressing discs}

We present general algorithms to decide whether any 3--manifold satisfying the hypotheses of \Cref{thm:Tollefson-singular,thm:some extremal is closed essential} contains a closed essential surface.

\subsection{Non-compact manifolds}
\label{subsec:knots algo}
We present the algorithm in two stages below.
\Cref{alg-incompressible} describes a subroutine to test
whether a given \emph{separating} closed surface is incompressible.
\Cref{alg-large} is the main algorithm: it uses the results
of \Cref{sec:closed} to identify candidate essential surfaces, and
runs \Cref{alg-incompressible} over each.

These algorithms contain a number of high-level and often intricate
procedures, many of which are described in separate papers.
For each algorithm, we discuss these procedures in further detail after
presenting the overall algorithm structure.

\begin{algorithm}[Testing for incompressibility of separating surface]
\label{alg-incompressible}
Suppose $\tri$ is known to be an ideal triangulation of the interior $M$ of an irreducible and $\partial$--irreducible, compact, orientable 3--manifold with non-empty boundary.
Let $S$ be a separating, closed, two-sided normal surface of genus $g \geq 1$
within $\tri$.  To test whether $S$ is incompressible in $M$:
\begin{enumerate}
    \item
    \label{en-alg-truncate}
    Truncate each ideal vertex of $\tri$ (i.e.\thinspace remove a small open
    neighbourhood of that vertex) to obtain a compact manifold
    with boundary, cut $\tri$ open along the surface $S$, and retriangulate.
    The result is a pair of triangulations $\tri_1,\tri_2$ representing
    two compact manifolds with boundary $M_1,M_2$ (one on each side of
    $S$ in $M$).

    Let $S_1,S_2$ be the genus $g$ boundary components of $\tri_1$ and
    $\tri_2$ respectively that correspond to the surface $S$, and let $B_k,$ $1 \le k \le |\partial \overline{M}|$ be the remaining boundary components of the triangulations.

    \item
    For each $i=1,2$:
    \begin{enumerate}
        \item \label{en-alg-simplify}
        Simplify $\tri_i$ into a triangulation with no internal
        vertices and only one vertex on each boundary component,
        without increasing the number of tetrahedra.
        Let the resulting number of tetrahedra in $\tri_i$ be $n$.

        \item \label{en-alg-search}
        Search for a connected normal surface $E$ in $\tri_i$
        that is not a vertex link, has positive Euler
        characteristic, and does not
        meet any of the boundary components $B_k$.

        \item \label{en-alg-nodisc}
        If no such $E$ exists, then there is no compressing disc
        for $S$ in $M_i$.  If $i=1$ then try $i=2$ instead, and if $i=2$
        then terminate with the result that $S$ is incompressible.

        \item \label{en-alg-crush}
        Otherwise
        crush the surface $E$
        as explained in \Cref{subsec:Crushing}
        to obtain a new triangulation $\tri'_i$ (possibly disconnected,
        or possibly empty) with strictly fewer than
        $n$ tetrahedra.  If some component of $\tri'_i$ has the same
        genus boundary (or boundaries) as $\tri_i$ then it
        represents the same manifold $M_i$, and we return to
        step~\ref{en-alg-simplify} using this component of $\tri'_i$ instead.
        Otherwise we terminate with the result that $S$ is not incompressible.
    \end{enumerate}
\end{enumerate}
\end{algorithm}

\noindent
Regarding the individual steps of this algorithm:
\begin{itemize}
    \item Step~\ref{en-alg-truncate} requires us to truncate an ideal
    vertex and cut a triangulation open along a normal surface.
    These are standard (though intricate) procedures.
    To truncate a vertex we subdivide tetrahedra and then remove the
    immediate neighbourhood of the vertex.
    To cut along a normal surface is more complex;
    a manageable implementation is described in \cite{burton12-ws}.

    \item Step~\ref{en-alg-simplify} requires us to simplify
    a triangulation to use the fewest possible vertices, without
    increasing the number of tetrahedra.
    For this we begin with the rich polynomial-time simplification
    heuristics described in \cite{burton12-regina}.  In practice, for all
    $2979$ knots that we consider in \Cref{s-results},
    this is sufficient to reduce the triangulation
    to the desired number of vertices.

    If there are still extraneous vertices, we can remove these using
    the crushing techniques of Jaco and Rubinstein
    \cite[Section~5.2]{jaco03-0-efficiency}.  This might
    fail, but only if  $\partial M_i$ has  a compressing disc,
    or two boundary components of $M_i$ are separated by a sphere; since $M$ is $\partial$--irreducible and irreducible
    both cases immediately certify that the surface $S$
    is compressible, and we can terminate immediately.

    \item Step~\ref{en-alg-search} requires us to locate a connected
    normal surface $E$ in $\tri_i$ that is not a vertex link, has
    positive Euler characteristic, and does not meet any of the boundary components $B_k$. 
    For this we use the recent method of
    \cite[Algorithm~11]{burton12-unknot}, which draws on
    combinatorial optimisation techniques:
    in essence we run a sequence of linear programs over a combinatorial
    search tree, and prune this tree using tailored branch-and-bound
    methods.  See \cite{burton12-unknot} for details.

    We note that this search is the bottleneck of
    \Cref{alg-incompressible}:
    the search is worst-case exponential time, though in practice it
    often runs much faster \cite{burton12-unknot}.
    The exposition in \cite{burton12-unknot} works in the setting where
    the underlying triangulation is a knot complement, but the
    methods work equally well in our setting here.
    To avoid the boundary components $B_k$, we simply remove all normal
    discs that touch any of the $B_k$ from our coordinate system.
\end{itemize}

\begin{theorem} \label{thm-incompressible}
    \Cref{alg-incompressible} terminates, and its output is correct.
\end{theorem}

\begin{proof}
    The algorithm terminates because each time we loop
    back to step~\ref{en-alg-simplify} we have fewer tetrahedra
    than before.  To prove correctness, we now devote
    ourselves to proving the many claims that are made throughout the
    statement of \Cref{alg-incompressible}.
    
    Before proceeding, however, we make a brief note regarding
    irreducibility.
    Since $M$ is irreducible,
    every embedded 2--sphere in $M$ must bound a ball.  As a result, the two manifolds $M_1$ and $M_2$
    are likewise irreducible, with the following possible exception. Suppose $M_j$ is reducible, so there is a sphere $F$ in $M_j$ which does not bound a ball. Since $M$ is irreducible, this sphere bounds a ball in $M$ and hence all boundary components $B_k$ of $M$ are on one side of this sphere. Therefore, they are all boundary components of $M_j$, and the sphere $F$ separates the union of all $B_k$ from $S_j.$ Whence $S$ is contained in a ball in $M$ and therefore compressible. (Note that a compression disc for $S$ may be in the reducible manifold $M_j$ or in the other component, which is necessarily irreducible.)
  
    We proceed now with proofs of the various claims made in
    \Cref{alg-incompressible}.

    \begin{itemize}
        \item \emph{In step~\ref{en-alg-truncate} we claim that cutting along
        $S$ yields two compact manifolds}.

        This is because $S$ is a assumed to be separating.

        \item \emph{In step~\ref{en-alg-nodisc} we claim that, if the surface
        $E$ cannot be found in $\tri_1$ \emph{and} it cannot be found in
        $\tri_2$, then the original surface $S$ must be incompressible.}

        Suppose that $S$ were compressible, with a compression disc
        in some $M_i$.  If this $M_i$ is irreducible,
        then by a result of Jaco and Oertel \cite[Lemma~4.1]{jaco84-haken}
        there is a \emph{normal} compressing disc in $\tri_i$.
        Since the underlying manifold is $\partial$--irreducible, this compressing disc
        must meet the genus $g$ boundary $S_i$ (not any of the  $B_k$), and so it is a surface of the type we are
        searching for.
        If this $M_i$ is reducible then (from earlier) we have that there
        is a properly embedded sphere within $M_i$, which separates the boundary components $B_k$ of $M_i$ from $S_i,$ so it is a surface of the type     we are searching for.
               
        \item \emph{In step~\ref{en-alg-crush} we claim that
        the new triangulation $\tri'_i$ has strictly fewer tetrahedra
        than $\tri_i$.}

        This is because $E$ is connected but not a vertex link, and
        therefore contains at least one normal quadrilateral.
        As noted in \Cref{subsec:Crushing},
        this means that at least one tetrahedron of $\tri_i$ is deleted
        in the Jaco-Rubinstein crushing process.

        \item \emph{In step~\ref{en-alg-crush} we claim that
        if $\tri'_i$ has a component with the same genus boundary
        (or boundaries) as $\tri_i$ then this component represents the same
        manifold $M_i$, and if not then $S$ is compressible.}

        Since the surface $E$ that we crush is connected with positive
        Euler characteristic and can be embedded within an irreducible, orientable 3--manifold with non-empty boundary, it follows that $E$ is either a    sphere or a disc.
        From \Cref{subsec:Crushing}, this means that when we crush $E$
        in the triangulation $\tri_i$, the resulting manifold is obtained
        from $M_i$ by a sequence of \emph{zero or more} of the following operations:
        \begin{itemize}
            \item undoing connected sums;
            \item cutting open along properly embedded discs;
            \item filling boundary spheres with 3--balls;
            \item deleting an entire component which is homeomorphic with a 3--ball, 3--sphere, $\R P^3$, $L_{3,1}$ or
            $S^2 \times S^1$.
        \end{itemize}
             Since $M$ is irreducible and $\partial$--irreducible, cutting $M_i$ open along a properly embedded disc either cuts off a 3--ball or corresponds to a compression of $S_i.$ We now analyse the effect of crushing on $M_i.$

        Suppose that $M_i$ is irreducible.  Then undoing a connected sum
        simply has the effect of creating an extra 3--sphere component (which will be deleted) and filling a boundary sphere with a ball.
        If we ever cut along a properly embedded disc that is \emph{not}
        a compressing disc, then likewise this just creates an extra
        3-ball component.  If we cut along a compressing disc, then
        this yields one or two pieces with strictly smaller total
        boundary genus than before; moreover, since the underlying manifold is $\partial$--irreducible, the first such compression must take place along
        the genus $g$ boundary $S_i$ (not any $B_k$) and so $S$ must be
        compressible. Together these observations establish the full
        claim above: we either terminate with the correctly identified result that $S$ is compressible or return a smaller triangulation of $M_i$ to step~\ref{en-alg-simplify}.

        Suppose $M_i$ is reducible. Then, as above, undoing a connected sum either creates extra 3--sphere components, or we have undone a non-trivial connected sum. In the latter case, one sphere in the associated collection must  separate boundary components of $M_i$, so that no component of $\tri'_i$ has the same boundary as $\tri_i,$ and since this sphere certifies that $S$ is compressible the conclusion is correct. Similarly, cutting along a 
        properly embedded disc that is \emph{not}
        a compressing disc,  creates an extra
        3-ball component, whilst cutting along a compression discs changes the boundary of $M_i.$
Whence in this case, we also either terminate with the correctly identified result that $S$ is compressible or return a smaller triangulation of $M_i$ to step~\ref{en-alg-simplify}.
    \end{itemize}
    This completes the proof of the theorem.
\end{proof}

We can now package together a full algorithm to test for closed essential surfaces:

\begin{algorithm}[Closed essential surface in ideal triangulation] \label{alg-large}
Suppose $\tri$ is known to be an ideal triangulation of the interior $M$ of an irreducible and $\partial$--irreducible, compact, orientable 3--manifold $\overline{M}$ with non-empty boundary.

    To test whether $M$ contains a closed essential surface:
    \begin{enumerate}
        \item \label{en-test-homology}
        Test whether $2b_1(M)>b_1(\partial \overline{M}).$ If yes, then  terminate
        with the result that $M$ contains a closed essential surface.

        \item \label{en-test-enumerate}
        Otherwise enumerate all extremal rays of $Q_0(\tri)$;
        denote these $\mathbf{e}_1,\ldots,\mathbf{e}_k$.
        For each extremal ray $\mathbf{e}_i$,
        let $S_i$ be the unique connected two-sided normal
        surface for which $x(S_i)$ lies on $\mathbf{e}_i$.
        From the previous step, we know that each $S_i$ is separating.

        \item \label{en-test-essential}
        For each non-spherical surface $S_i$,
        use \Cref{alg-incompressible} to test whether $S_i$
        is incompressible in $\tri$. If the genus of $S_i$ is different from the genera of the vertex links of $M$, 	then terminate with the result that $M$ contains a closed essential surface.
        
        \item \label{en-test-special}
        Now each $S_i$ is a sphere or has genus equal to a vertex link.
        So for each non-spherical surface $S_i$, 
         test whether
        $S_i$ is boundary parallel by (i)~cutting $\tri$ open along $S_i$ and truncating all ideal vertices,
        and then (ii)~using the Jaco-Tollefson algorithm
        \cite[Algorithm~9.7]{jaco95-algorithms-decomposition}
        to test whether one of the resulting components is homeomorphic to the product
        space $S_i \times [0,1]$.  If $S_i$ is not boundary parallel, then terminate
        with the result that $M$ contains a closed essential surface. Otherwise all  incompressible surfaces are
        found to be boundary parallel, then terminate with the result that $M$ contains no closed essential surface.     
        \end{enumerate}
\end{algorithm}

\noindent
Regarding the individual steps:
\begin{itemize}
    \item Step~\ref{en-test-homology} requires us to compute homology. This is a standard routine using Smith normal form, and implemented following Hafner and McCurley~\cite{hafner91-snf}.
    \item Step~\ref{en-test-enumerate} requires us to enumerate all
    extremal rays of $Q_0(\tri)$.  This is an expensive procedure
    (which is unavoidable, since there is a worst-case exponential
    number of extremal rays).
    For this we use the recent state-of-the-art tree traversal method
    \cite{burton13-tree}, which is tailored to the constraints and
    pathologies of normal surface theory and is found to be highly
    effective for larger problems.
    The tree traversal method works in the
    larger cone $Q(\tri)$, but it is a simple matter to insert the
    additional linear equations corresponding to $\nu_x=0$.

    We also note that it is simple to identify the unique
    closed two-sided normal surface for which $x(S)$ lies on the
    extremal ray $\mathbf{e}$.  Specifically, $x(S)$ is either
    the smallest integer vector on $\mathbf{e}$ or, if that vector
    yields a one-sided surface, then its double.

    \item Step~\ref{en-test-special} requires
    us to run the Jaco-Tollefson algorithm to test whether any
    incompressible surface is boundary-parallel.
    This algorithm is expensive: it requires us to work in a larger
    normal coordinate system, solve difficult enumeration
    problems, and perform intricate geometric operations. However, in our applications, we work (mostly) with hyperbolic 3--manifolds of finite volume, for which we never enter this step. 
Also, in the case where the boundary of the manifold consists of tori, there are additional fast methods for avoiding the Jaco-Tollefson algorithm. For instance, one may run Haraway's $T^2\times I$ test \cite[Proposition 13]{haraway-determining-2014} in conjunction with the algorithms from \cite{burton12-unknot}.
\end{itemize}

\begin{theorem} \label{thm-large}
    \Cref{alg-large} terminates, and its output is correct.
\end{theorem}

\begin{proof}
    The algorithm terminates because it does not contain any loops.
    All that remains is
    to prove that its output is correct.

    Throughout this proof we
    implicitly use \Cref{thm-incompressible} to verify that
    all calls to \Cref{alg-incompressible} are themselves correct.
    We note that the conditions of
    \Cref{thm:some extremal is closed essential} apply.
    
    From \Cref{thm:some extremal is closed essential},
    the manifold $M$ contains a closed essential surface if and only if one of the closed normal
    surfaces $S_i$ in our list (excluding spheres) is essential. We ignore spheres from now onwards.
    We note that each surface in the list that has genus different from all vertex links is essential
    if and only if it is incompressible (since such a surface cannot be
    boundary parallel),
    and each remaining surface in the list is essential if and only if it is
    (i)~incompressible and (ii)~not boundary parallel.

    Steps~\ref{en-test-essential} and \ref{en-test-special} of
    \Cref{alg-large} test precisely these conditions, and so
    the algorithm is correct.  The only reason for the order in
    these steps is so that we can use \Cref{alg-incompressible}
    exclusively if possible, and only fall through to the more expensive
    Jaco-Tollefson algorithm when absolutely necessary.
\end{proof}

\subsection{Closed manifolds}
\label{subsec:HW algorithm}

For reference, we also spell out the algorithm for closed, irreducible, orientable 3--manifolds, which follows the same outline as the algorithm for non-compact 3--manifolds, but with fewer steps. Indeed, cutting along a separating normal surface in a closed 3--manifold results in two compact 3--manifolds which have a copy of this surface as a boundary component. Components of this kind are also dealt with by \Cref{alg-incompressible}.

\begin{algorithm}[Closed essential surface in triangulation] \label{alg-closed-Haken}
Suppose $\tri$ is known to be a (possibly singular) triangulation of the closed, irreducible, orientable 3-manifold $M.$ 
To test whether $M$ contains a closed essential surface:
    \begin{enumerate}
        \item \label{closed-test-homology}
        Test whether $b_1(M)>0.$ If yes, then terminate
        with the result that $M$ contains a closed essential surface.

        \item \label{closed-test-enumerate}
        Otherwise enumerate all extremal rays of $Q(\tri)$;
        denote these $\mathbf{e}_1,\ldots,\mathbf{e}_k$.
        For each extremal ray $\mathbf{e}_i$,
        let $S_i$ be the unique connected two-sided normal
        surface for which $x(S_i)$ lies on $\mathbf{e}_i$.
        From the previous step, we know that each $S_i$ is separating.

        \item \label{closed-test-essential}
        For each non-spherical surface $S_i$,
        cut $\tri$ open along $S_i$ and retriangulate. The result is a pair of triangulations $\tri_1$ and $\tri_2$ representing two compact manifolds with boundary $M_1$ and $M_2.$ Let $S_j = \partial M_j.$ For each $i=1,2:$
         \begin{enumerate}
		\item \label{closed-alg-simplify} 
        Simplify $\tri_i$ into a triangulation with no internal
        vertices and only one vertex on its boundary component,
        without increasing the number of tetrahedra.
        Let the resulting number of tetrahedra in $\tri_i$ be $n$.

        \item \label{closed-alg-search}
        Search for a connected normal surface $E$ in $\tri_i$
        that is not a vertex link and has positive Euler
        characteristic.

        \item \label{closed-alg-nodisc}
        If no such $E$ exists, then there is no compressing disc
        for $S_i$ in $M_i$.  If $i=1$ then try $i=2$ instead, and if $i=2$
        then terminate with the result that $S_i$ is incompressible.
              \item \label{closed-alg-crush}
        Otherwise
        crush the surface $E$
        as explained in \Cref{subsec:Crushing}
        to obtain a new triangulation $\tri'_i$ (possibly disconnected,
        or possibly empty) with strictly fewer than
        $n$ tetrahedra.  If some component of $\tri'_i$ has the same
        genus boundary as $\tri_i$ then it
        represents the same manifold $M_i$, and we return to
        step~\ref{closed-alg-simplify} using this component of $\tri'_i$ instead.
        Otherwise we terminate with the result that $S$ is not incompressible.		
        \end{enumerate}
        If no non-spherical surface $S_i$ has been found to be incompressible, terminate with the result that $M$ contains no closed essential surface.
        \end{enumerate}
\end{algorithm}

The arguments showing that \Cref{alg-closed-Haken} is correct and terminates are analogous to the ones given for \Cref{alg-incompressible} and \ref{alg-large}, and will therefore not be stated again.

\section{Algorithm engineering and implementation}

Since \Cref{alg-large,alg-closed-Haken} have
doubly-exponential running times in theory, they must be implemented
with great care if we are to hope for running times that are
nevertheless feasible in practice.

Some of this comes down to ``traditional'' algorithm engineering---careful
choices of data structures and code flow that make frequent
operations very fast (often using trade-offs that make less
frequent operations slower).  Such techniques are common practice in
software development, and we do not discuss them further here.

What is more important, however, is to implement the algorithms in
such a way that:
\begin{itemize}
    \item if the input to some expensive procedure is
    ``well-structured'' in a way that allows an answer to be seen quickly,
    then the procedure can identify this and terminate early;
    \item if the input is \emph{not} well-structured, then the code
    attempts to find an equivalent input (e.g., a different triangulation
    of the same manifold) that \emph{does} allow early termination as
    described above.
\end{itemize}

For \Cref{alg-large,alg-closed-Haken}, the most
expensive procedure is in step~3 of each algorithm,
where we cut the original manifold open along a normal surface $S_i$
and attempt to either (i)~find a compressing disc in one of the
resulting triangulations $\mathcal{T}_1$ or $\mathcal{T}_2$, or
(ii)~certify that no such compressing disc exists.

For this procedure, the implementations in {\regina} are designed as
follows:

\begin{itemize}
    \item When searching for a compressing disc in each triangulation
    $\mathcal{T}_j$, we begin by optimistically checking
    ``simple'' locations in which
    compressing discs are often found \cite{burton12-ws}.

    This includes looking for discs formed from a single face of the
    triangulation whose edges all lie in the boundary,
    or discs that slice through a single tetrahedron encircling an edge
    of degree~one (\Cref{simplediscs}).

    \begin{figure}[ht]
    \centering
    \includegraphics[scale=0.5]{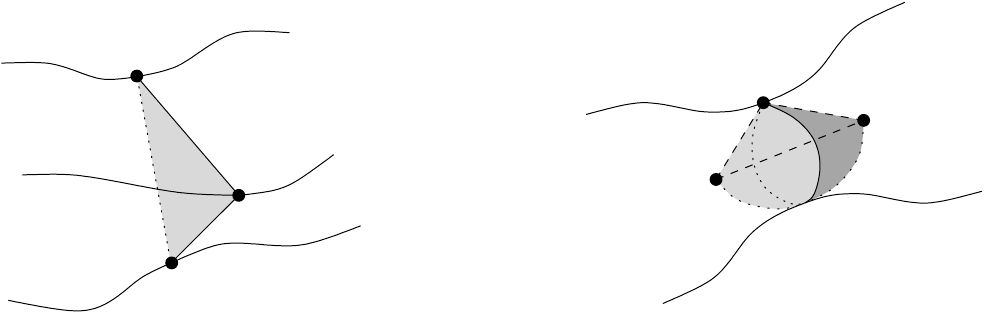}
    \caption{Examples of simple compressing discs} \label{simplediscs}
    \end{figure}

    Such discs are fast to locate and check, and only if no such ``simple''
    discs are found do we fall back to a full search 
    (i.e., a branch-and-bound search for a normal surface with
    positive Euler characteristic, as described earlier in this paper).

    \item After cutting along $S_i$ to obtain the pair of triangulations
    $\mathcal{T}_1$ and $\mathcal{T}_2$, we search for a compressing disc
    in both triangulations $\mathcal{T}_1$ and $\mathcal{T}_2$
    \emph{simultaneously}.

    This allows the procedure to terminate early in the case where $S_i$
    is compressible, since if we find a compressing disc in one
    triangulation $\mathcal{T}_j$ then there is no need to finish
    searching the other.

    \item Extending this idea of parallel processing further:
    when searching each triangulation $\mathcal{T}_j$ for a compressing
    disc, we simultaneously search through several different
    \emph{retriangulations} of $\mathcal{T}_j$ (i.e., different
    triangulations of the same manifold with boundary).

    As before, as soon as any of these retriangulations yields a compressing
    disc then we can immediately terminate the entire procedure for the
    surface $S_i$ (which is now known to be compressible).

    Moreover, this parallelisation also helps in the case where
    $S_i$ is \emph{incompressible}.  It is often found in practice that,
    when searching for a compressing disc, the branch-and-bound search
    for a positive Euler characteristic surface finishes in remarkably
    few steps \cite{burton12-unknot}.
    If this happens with \emph{any} of our retriangulations
    then we will have certified that $\mathcal{T}_j$ does not
    contain a compressing disc, and we can immediately stop processing
    $\mathcal{T}_j$ and instead devote our attention to the manifold on
    the other side of $S_i$.

    Since the performance of the branch-and-bound code in practice depends
    heavily on the combinatorial structure of the underlying
    triangulation, it is important to choose retriangulations of
    $\mathcal{T}_j$ that are as \emph{dissimilar} as possible.
    For this reason, the implementation in {\regina} creates these
    retriangulations immediately after cutting along $S_i$, before
    performing any simplification moves---this increases the chances
    that different retriangulations of $\mathcal{T}_j$ remain
    substantially different even after they are subsequently simplified.
    (In the worst case, if two retriangulations simplify to become
    combinatorially \emph{identical}, then we discard the duplicate and
    attempt yet another retriangulation to take its place.)

    \item When testing the candidate incompressible surfaces $S_i$ in step~3
    of \Cref{alg-large,alg-closed-Haken}, we work
    through these surfaces in order from smallest genus to largest.

    This is because higher genus surfaces are likely to result in larger
    triangulations $\mathcal{T}_j$, which could potentially mean
    \emph{much} longer running times (since testing compressibility
    requires exponential time in the size of $\mathcal{T}_j$).

    If there is no incompressible surface then the total running time is
    not affected (since every surface must be checked regardless).
    If there \emph{is} an incompressible surface, however, then
    processing the surfaces in order by genus makes it more likely that
    an incompressible surface is found before the triangulations
    $\mathcal{T}_j$ become too large to handle.
\end{itemize}

We note that \Cref{alg-large} contains another potentialy
expensive step: testing whether a surface $S_i$ is boundary parallel,
which involves cutting along $S_i$ and testing whether either side forms the
product $S_i \times [0,1]$.
This procedure is \emph{not} optimised in {\tt Regina}\rm, because (for this
paper at least) it does not matter---in every knot complement that we
processed,
\emph{none} of the candidate surfaces $S_i$ were tori (and therefore
none were boundary parallel).

The discussion above only outlines the major optimisations in the
implementation of \Cref{alg-large,alg-closed-Haken}.
For further details, the reader is encouraged to read through the
thoroughly documented source code in {\tt Regina}\rm~\cite{regina}.

\section{Computational results} \label{s-results}

This section gives some additional information about the computational results. The complete data are available at \cite{regina}.


\subsection{Closed manifolds}
\label{subsec:HW census}

The Hodgson-Weeks census contains 11,031 closed, orientable 3--manifolds.
The theoretical running time of our algorithm is $\exp(\exp(O(n))),$ where $n$ is the number of tetrahedra. As stated in the introduction, the number of tetrahedra ranges from $9$ to $32$ over the census. \Cref{fig:hw-time} plots the running times (measuring wall time) for enumerating the candidate surfaces and deciding incompressibility.

\begin{figure}[h]
\begin{center}
  \includegraphics[width=7cm]{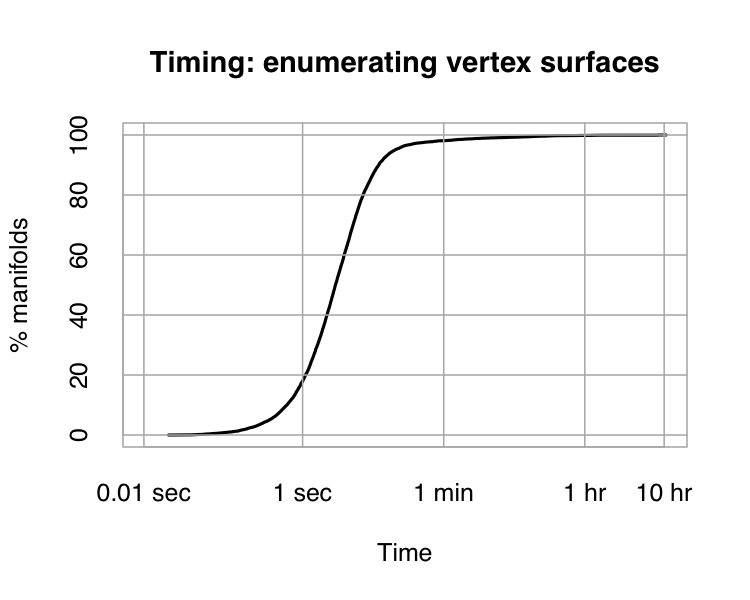}
  \qquad
  \includegraphics[width=7cm]{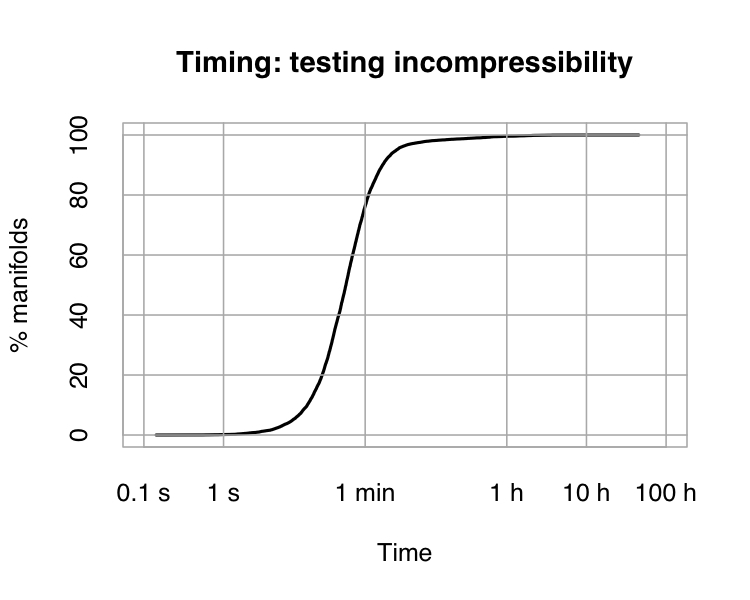}
\end{center}
\caption{Timing for the Hodgson-Weeks census}
\label{fig:hw-time}
\end{figure}

Given a vertex normal surface $S$ in a triangulation of $M$ with $n$ tetrahedra we expect order of $\exp(n)$ tetrahedra in a triangulation of $M\comp S,$ so between $10^4$ and $10^{14}$ tetrahedra after cutting and retriangulating for the census manifolds. \Cref{fig:hw-sliced} shows that in practice the numbers are magnitutes smaller.

\begin{figure}[h]
\begin{center}
  \includegraphics[width=8cm]{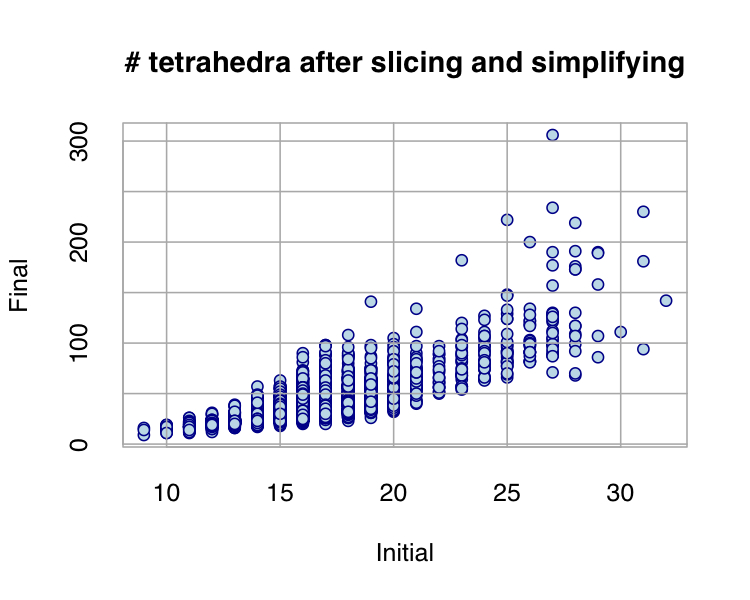}
\end{center}
\caption{Number of tetrahedra after slicing and simplifying}
\label{fig:hw-sliced}
\end{figure}

A plot showing the percentage of Haken manifolds in the census is shown in \Cref{fig:hw-percent} in the introduction.


\subsection{Knot complements}
\label{subsec: knots computation}

The census of all knots in the 3--sphere with at most 14 crossings due to Hoste-Thistlethwaite-Weeks~\cite{hoste98-first} contains 59,924 hyperbolic knots with at most 14 crossings, and the number of ideal tetrahedra used to triangulate their complements ranges from 2 to 33. \Cref{fig:htw-tvc} gives an idea of the number of tetrahedra versus the number of crossings.
\begin{figure}[h]
\begin{center}
  \includegraphics[width=6cm]{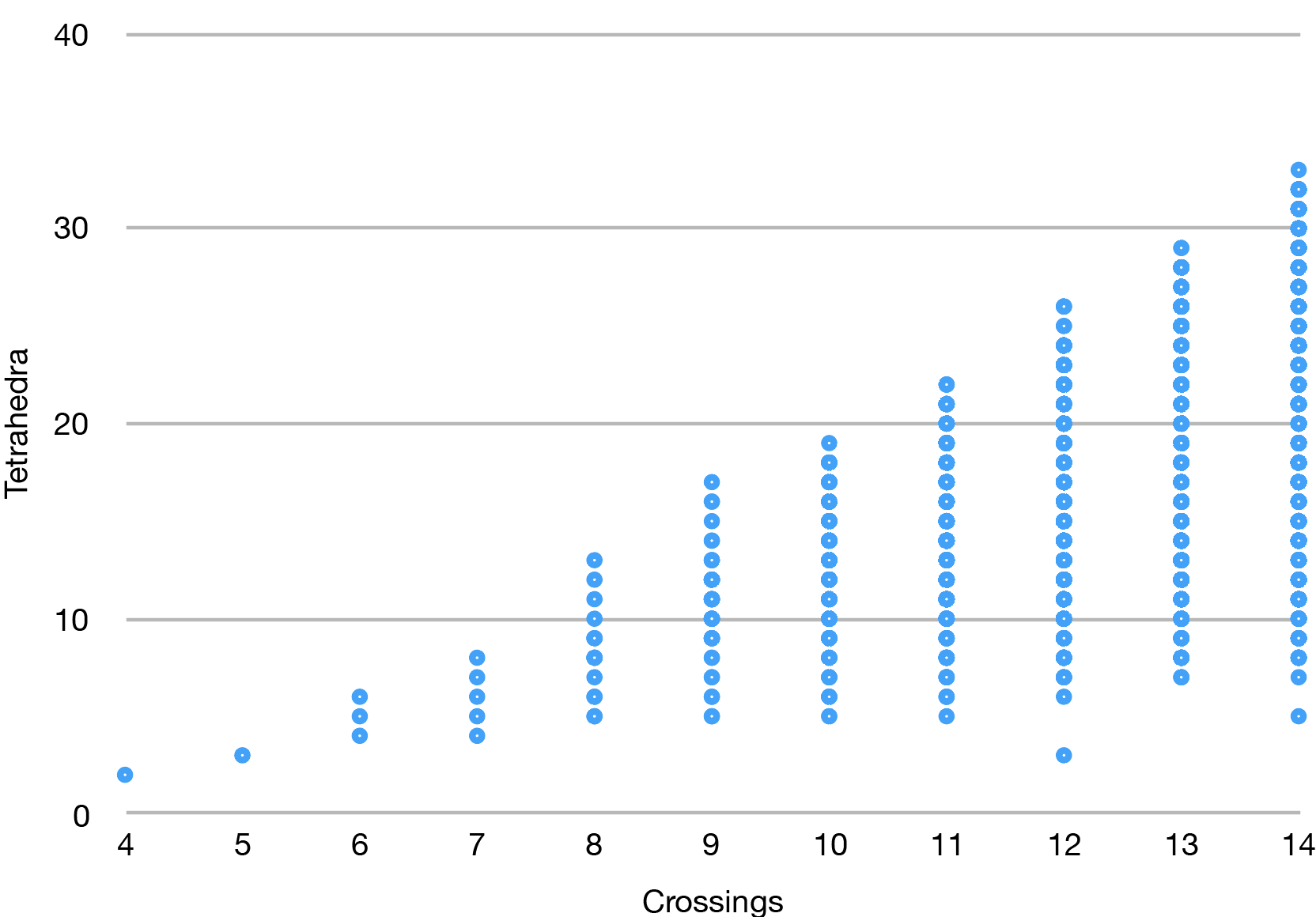}
   \qquad\qquad
   \includegraphics[width=6cm]{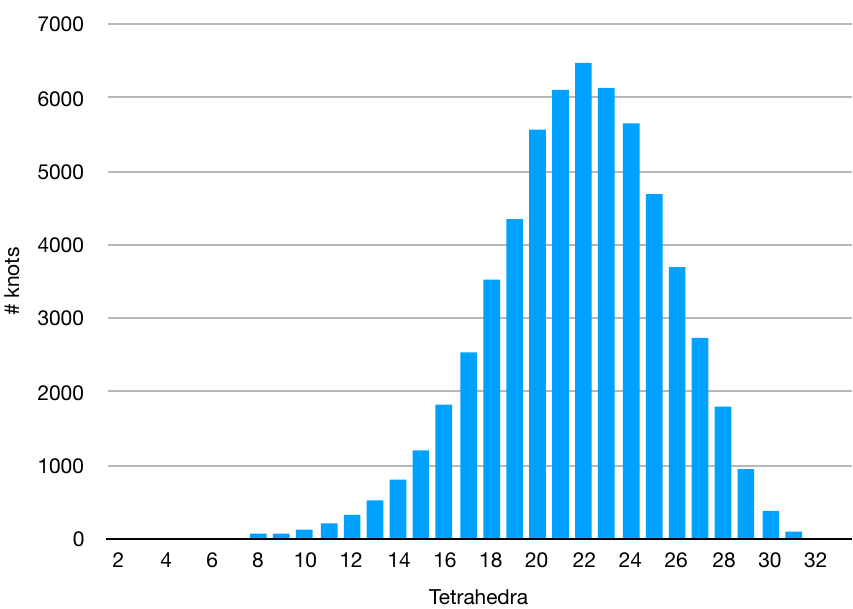}
\end{center}
\caption{Number of tetrahedra versus number of crossings}
\label{fig:htw-tvc}
\end{figure}

\Cref{fig:htw-time} plots the running times (measuring wall time) for enumerating the candidate surfaces and deciding incompressibility.

\begin{figure}[h]
\begin{center}
  \includegraphics[width=6cm]{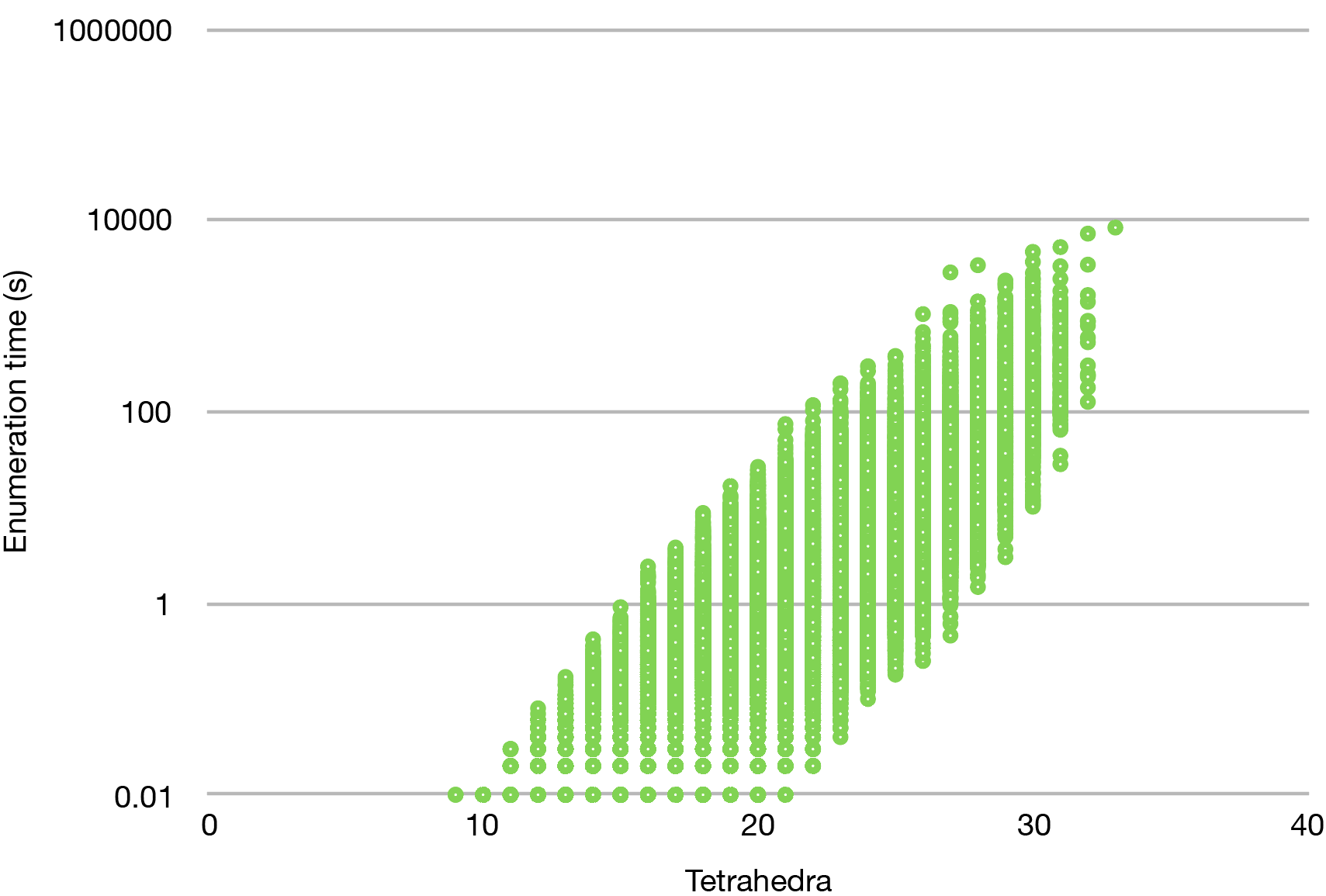}
  \qquad\qquad
  \includegraphics[width=6cm]{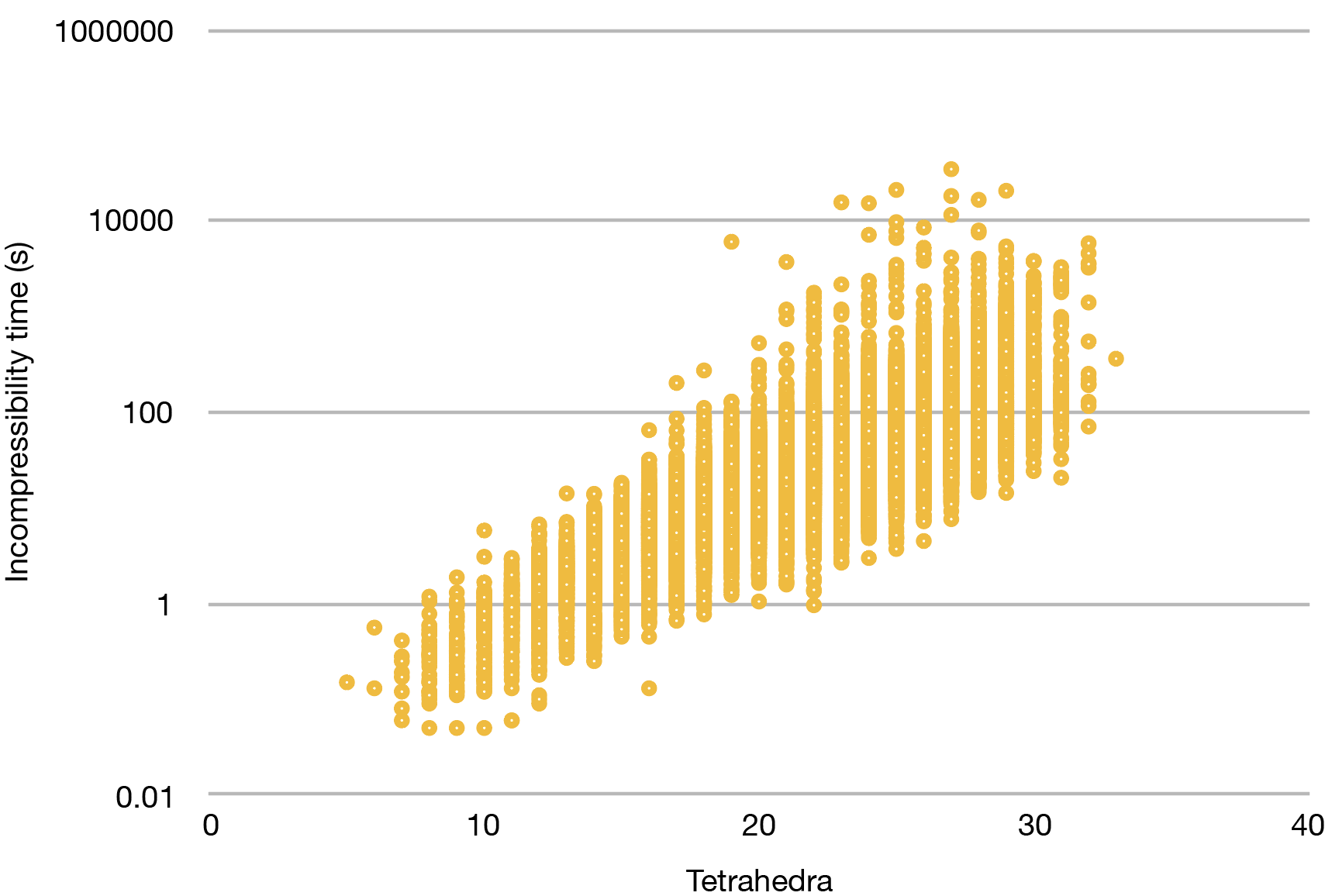}
\end{center}
\caption{Timing for the Hoste-Thistlethwaite-Weeks census}
\label{fig:htw-time}
\end{figure}

We give three different perspectives of the percentage of knot diagrams containing closed essential surfaces (colloquially called \emph{large knots}). 
Namely, by number of crossings (\Cref{fig:htw-bycross}), by number of tetrahedra (\Cref{fig:htw-bytet}) and by volume (\Cref{fig:htw-byvol}). 

\begin{figure}[h!]
    \centering
     \subfigure[Pointwise]{%
        \label{fig:htw-bycross-pointwise}%
  \includegraphics[width=6cm]{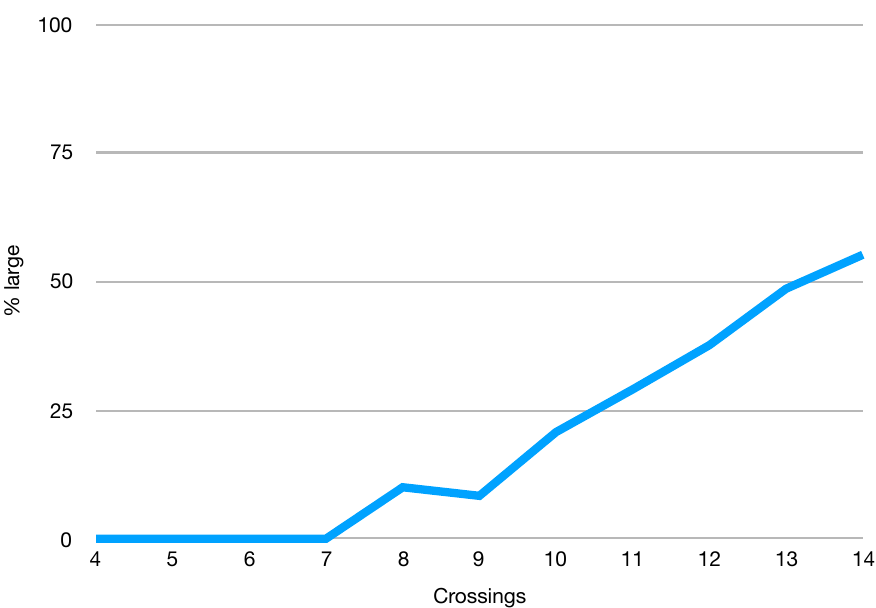}}
  \qquad\qquad
      \subfigure[Cumulative]{%
        \label{fig:htw-bycross-cumulative}%
  \includegraphics[width=6cm]{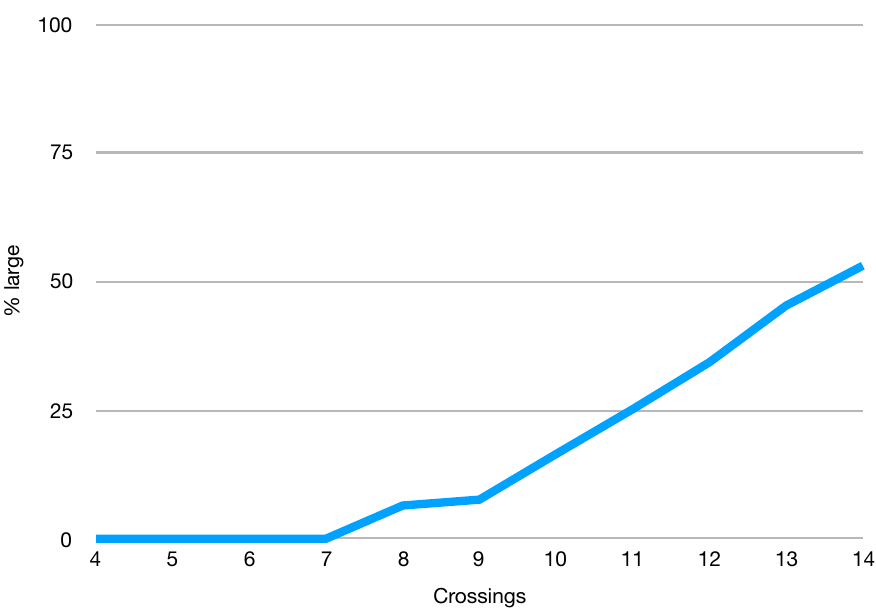}}
\caption{Percentage of large knots by number of crossings}
\label{fig:htw-bycross}
\end{figure}

\begin{figure}[h]
    \centering
     \subfigure[Pointwise]{%
        \label{fig:htw-bytet-pointwise}%
  \includegraphics[width=6cm]{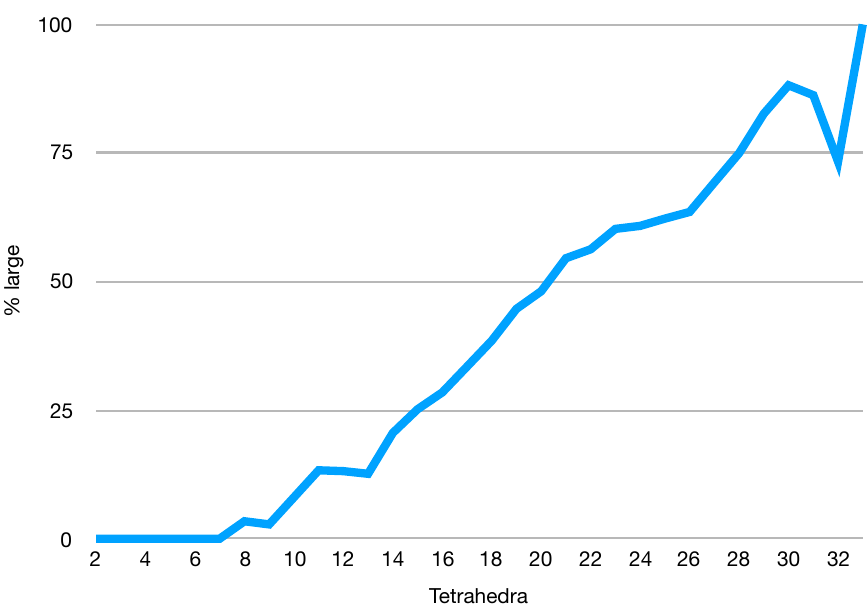}}
  \qquad\qquad
      \subfigure[Cumulative]{%
        \label{fig:htw-bytet-cumulative}%
  \includegraphics[width=6cm]{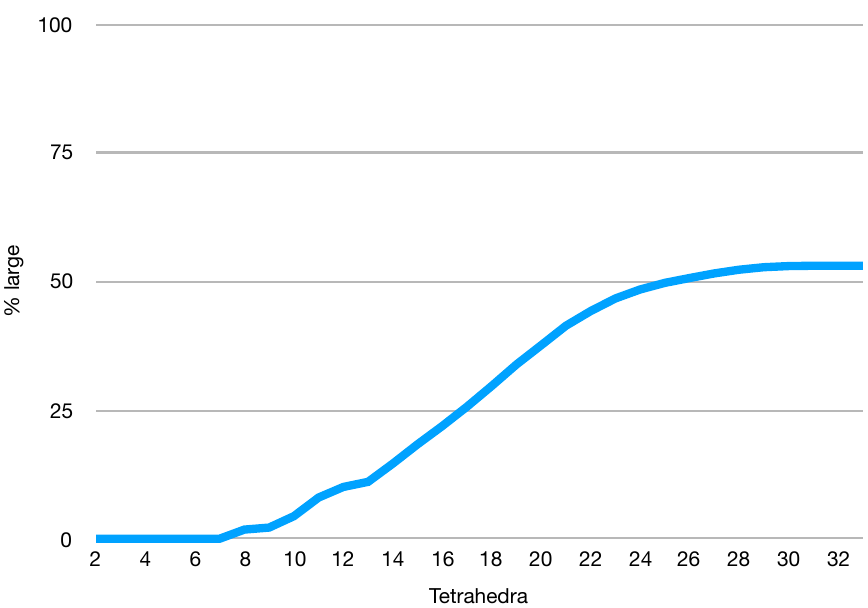}}
\caption{Percentage of large knots by number of tetrahedra}
\label{fig:htw-bytet}
\end{figure}

We remark that for large numbers of tetrahedra, there are only few knot complements (see \Cref{fig:htw-tvc}), which explains in \Cref{fig:htw-bytet} both the levelling out of the cumulative plot as well as the jumps in the pointwise plot.


\subsection*{Acknowledgements} 

The authors are  supported  by  the  Australian  Research  Council  under  the  Discovery  Projects funding scheme (DP150104108 and DP160104502, respectively). On behalf of all authors, the corresponding author states that there is no conflict of interest.

%
%

\bibliographystyle{amsplain}
\bibliography{pure}

\bigskip

%
%

\address{Benjamin A. Burton\\School of Mathematics and Physics,\\The University of Queensland,\\QLD 4072 Australia\\(bab@maths.uq.edu.au)\\---}

\address{Stephan Tillmann,\\ School of Mathematics and Statistics F07,\\ The University of Sydney,\\ NSW 2006 Australia\\
(stephan.tillmann@sydney.edu.au)}

\Addresses

\end{document}